\documentclass[leqno,11pt,a4paper]{amsart}

\usepackage{amsmath}
\usepackage{amsthm}
\usepackage{cite}
\usepackage{enumitem}
\usepackage{hyperref}
\usepackage{graphicx}
\usepackage{amssymb}
\usepackage{xcolor}
\hypersetup{
  colorlinks,
  linkcolor={blue!50!black},
  citecolor={blue!50!black},
  urlcolor={blue!80!black}
}
\usepackage{multirow}
\usepackage[all]{xy}
\usepackage{tikz}
\usepackage{booktabs}
\usepackage{arydshln}
\usepackage{mleftright}
\usepackage{mathrsfs}

\setlist[enumerate]{labelsep=*, leftmargin=1.5pc}
\setlist[enumerate]{label=\normalfont(\roman*), ref=\roman*}
\usepackage[margin=2.7cm]{geometry}
\graphicspath{{images/}}
\theoremstyle{plain}
\newtheorem{thm}{Theorem}[section]
\newtheorem{pro}[thm]{Proposition}
\newtheorem{lem}[thm]{Lemma}

\newtheorem{conjecture}[thm]{Conjecture}
\theoremstyle{definition}
\newtheorem{dfn}[thm]{Definition}
\newtheorem{rem}[thm]{Remark}
\newtheorem{eg}[thm]{Example}

\newtheorem{cons}[thm]{Construction}

\DeclareMathOperator{\Pic}{Pic}

\DeclareMathOperator{\Gr}{Gr}
\DeclareMathOperator{\Crit}{Crit}

\DeclareMathOperator{\Ann}{Ann}

\DeclareMathOperator{\Vol}{Vol}
\DeclareMathOperator{\Edges}{Edges}
\DeclareMathOperator{\Faces}{Faces}

\DeclareMathOperator{\sgn}{sgn}

\DeclareMathOperator{\coker}{coker}
\DeclareMathOperator{\join}{Join}
\DeclareMathOperator{\Bl}{Bl}
\DeclareMathOperator{\Aut}{Aut}

\newcommand{\cD}{\mathcal{D}}
\newcommand{\cF}{\mathcal{F}}
\newcommand{\cG}{\mathcal{G}}

\newcommand{\cN}{\mathcal{N}}
\newcommand{\cO}{\mathcal{O}}

\newcommand{\CC}{\mathbb{C}}

\newcommand{\kk}{\Bbbk}

\newcommand{\PP}{{\mathbb{P}}}
\newcommand{\QQ}{{\mathbb{Q}}}
\newcommand{\RR}{{\mathbb{R}}}

\newcommand{\ZZ}{{\mathbb{Z}}}

\renewcommand{\P}{\mathscr{P}}

\newcommand{\Cstar}{\CC^\times}

\newcommand{\V}[1]{\operatorname{Verts}\mleft({#1}\mright)}
\newcommand{\LS}{\mathcal{LS}^+_{\text{pre},X}}
\newcommand{\KN}{\text{KN}}

\begin{document}
\author[T.\,Prince]{Thomas Prince}
\address{Mathematical Institute\\University of Oxford\\Woodstock Road\\Oxford\\OX2 6GG\\UK}
\email{thomas.prince@magd.ox.ac.uk}

\title[Calabi--Yau $3$-folds via the Gross--Siebert algorithm]{Smoothing Calabi--Yau toric hypersurfaces using the Gross--Siebert algorithm}
\maketitle
\begin{abstract}
	We explain how to form a novel dataset of simply connected Calabi--Yau threefolds via the Gross--Siebert algorithm. We expect these to degenerate to Calabi--Yau toric hypersurfaces with certain Gorenstein (not necessarily isolated) singularities. In particular, we explain how to `smooth the boundary' of a class of $4$-dimensional reflexive polytopes to obtain a polarised tropical manifolds. We compute topological invariants of a compactified torus fibration over each such tropical manifold, expected to be homotopy equivalent to the general fibre of the Gross--Siebert smoothing. We consider a family of examples related to the joins of elliptic curves. Among these we find $14$ topological types with $b_2=1$ which do not appear in existing lists of known rank one Calabi--Yau threefolds.
\end{abstract}

\section{Introduction}
\label{sec:introduction}

Calabi--Yau threefolds, three dimensional compact K\"{a}hler manifolds $X$ such that $h^1(X,\cO_X)$ and $h^2(X,\cO_X)$ vanish, are intensively studied objects in both algebraic geometry and theoretical physics. Datasets of such objects have been studied since the work of Candelas--Lutken--Schimmrigk \cite{CLS88} and  Candelas--Lynker--Schimmrigk~\cite{CLS90}. In \cite{B94} Batyrev described a construction of a Calabi--Yau threefold from any triangulation of a four dimensional reflexive polytope; extended by Batyrev--Borisov~\cite{BB96} to \emph{nef partitions} of higher dimensional reflexive polytopes. Together with the classification of $4$-dimensional reflexive polytopes by Kreuzer--Skarke \cite{KS00}, this construction provides an enormous number of Calabi--Yau threefolds. By way of illustration there are $473,800,776$ $4$-dimensional reflexive polytopes, without taking into account the number of triangulations.

Despite the plethora of Calabi--Yau threefolds obtained by the methods above, such lists do not necessarily imply an abundance of Calabi--Yau threefolds in a particular class. For example, \cite{K15} contains a -- then complete -- list of $151$ known constructions of Calabi--Yau threefolds of Picard rank one. In fact $20$ of these constructions are conjectural, and we explore the question of the existence of several such examples in \S\ref{sec:examples}, in light of the recent work of Inoue~\cite{I19} and Knapp--Sharpe~\cite[\S$2.5$]{KS}. In a different direction, a list of constructions of Calabi--Yau threefolds with small Hodge numbers was compiled by Candelas--Constantin--Mishra~\cite{CCM18}. In this article we describe an algorithm to construct a large class of new examples with small Picard rank, and generate examples. We construct (assuming Conjecture~\ref{conj:KN_spaces}) $14$ new topological types of simply connected Calabi--Yau threefolds with $b_2=1$. One of these topological types is predicted by the existence of an Calabi--Yau differential operator with integral monodromy~\cite{ES03} and appears as one of the $20$ conjectural examples listed in \cite{K15}.

The constructions we present in this article are based on the \emph{Gross--Siebert program}. This is an algebro-geometric approach to the Strominger--Yau--Zaslow (SYZ) conjecture, developed by Gross and Siebert in \cite{GS06,GS10,Gross--Siebert}. The main objects of study in this program are \emph{toric log Calabi--Yau spaces}; unions of toric varieties equipped with sections of line bundles which, by the results of \cite{GS06}, determine a log structure. These spaces determine \emph{integral affine manifolds with singularities}, which play a key role in the work of Gross, Haase--Zharkov, and Ruan, on topological versions of the SYZ conjecture; we refer to \cite{G01,G05,G01I,R07,HZ05} for further details.

Before stating our main result we recall that, given a $4$-dimensional reflexive polytope $P$, there is a canonical bijection $F \mapsto F^\star$ between $d$-dimensional faces of $P$ and $(3-d)$-dimensional faces of $P^\circ$, the polar polytope to $P$. We write $\ell(E)$ for the lattice length of a $1$-dimensional lattice polytope $E$.

\begin{thm}
	\label{thm:main_cons}
	Let $P$ be a $4$-dimensional reflexive polytope, and $D$ be a choice of Minkowski decomposition of each of the polygons in
	\[
	\left\{\ell(F^\star)F : F \in \Faces(P,2)\right\}
	\]
	into polygons affine linearly isomorphic to standard simplices (of dimensions $1$ and $2$). The pair $(P,D)$ determines a locally rigid, positive, toric log Calabi--Yau space $X_0(P,D)$. Moreover, the toric log Calabi--Yau space $X_0(P,D)$ admits a polarisation if $(P,D)$ is \emph{regular}; see Definition~\ref{dfn:regular}.
\end{thm}

It follows from \cite[Theorem~$1.30$]{Gross--Siebert} that, if $(P,D)$ is regular, $X_0(P,D)$ is the central fibre of a formal degeneration of log Calabi--Yau manifolds. Moreover, following recent work \cite[Theorem~$4.4$]{RS:draft} of Ruddat--Siebert, this degeneration may be extended to a family over a disc in $\CC$. We also expect our construction to be compatible with very recent work of Felten--Filip--Ruddat~\cite{FFR19} on the smoothability of toroidal crossing spaces. Indeed, we expect that the results of \cite{FFR19} will allow us to replace the condition of local rigidity required to apply the results of \cite{Gross--Siebert} with a more geometric argument, working only with the degeneration of a Calabi--Yau toric hypersurface to the toric boundary of the ambient toric variety (which is generally not locally rigid). If completed this would also suggest a direct link to work of Lee~\cite{L17,L18} on smoothing normal crossings Calabi--Yau varieties. The spaces we smooth in our construction are typically not normal crossings, but it is likely that there is significant overlap in the sets of Calabi--Yau threefolds obtained by these methods. 

\begin{rem}
	To prove Theorem~\ref{thm:main_cons}, we construct an integral affine manifold with singularities $B$, together with a polyhedral decomposition $\P$ of $B$ from $(P,D)$. To form input data to the Gross--Siebert algorithm from this we also require a \emph{polarisation} on $B$. This is the integral affine analogue of a strictly convex piecewise linear function. The existence of such a function for our decomposition imposes an obstruction, determined by the geometry of $P$ and $D$, on our ability to use the Gross--Siebert algorithm to form a smoothing. We propose that this condition captures a genuine geometric obstruction to smoothing the singularities of a toric Calabi--Yau hypersurface.
\end{rem}


In the second part of this article we study the topology general fibre of such a family. Precisely, we should analyse the \emph{Kato--Nakayama space} associated to the toric log Calabi-Yau $X_0(P,D)$; however, we replace this analysis with the topological model -- which we refer to as $X(P,D)$ -- introduced by Gross in \cite{G01I,G01}. It is a long-standing conjecture that these two spaces are homotopy equivalent (after fixing a \emph{phase} for the Kato--Nakayama space) and this is not expected to be difficult in dimension three.

\begin{conjecture}
	\label{conj:KN_spaces}
	Points of the Kato--Nakayama space $X^{\KN}$ with fixed phase form a topological space homotopy equivalent to $X(P,D)$.
\end{conjecture}

We then compute topological invariants of a $X(P,D)$. Assuming Conjecture~\ref{conj:KN_spaces}, these are the topological invariants of the general fibre of the Gross--Siebert family.

\begin{thm}
	\label{thm:invariants}
	The $6$-manifold $X(P,D)$ is simply connected. The topological Euler number $\chi(X(P,D))$ is equal to
	\[
	\sum_{F \in \Faces(P,2)}{\big(\#\{Q \in D(F) : \dim(Q) = 2\}\big)} - \sum_{G \in \Faces(P^\circ,2)}{\big(\ell(G^\star)^2\Vol(G)\big)},
	\]
	where $D(F)$ denotes the Minkowski decomposition of $F$ determined by $D$. The Betti number $b_2(X(P,D))$ is equal to $\gamma(P,D) - 3$; where $\gamma(P,D) = \dim \Gamma(P,D)$ is described in Definition~\ref{dfn:gammaPD}.
\end{thm}

While the classification of all possible input data to Theorem~\ref{thm:main_cons} is expected to be computationally accessible, we defer such a computation to future work. However, we present a family of examples in \S\ref{sec:examples}. In each of these $P$ is a product of reflexive polygons, and this family combines a number of new examples with a number of classical cases. Following observations of Galkin~\cite{G15}, and constructions of Inoue~\cite{I19} and Knapp--Sharpe~\cite{KS19}, many of these examples are related to \emph{joins} of elliptic curves. Indeed, we expect all smoothing components of all joins of pairs of anti-canonical sections in del~Pezzo surfaces of degree at least $3$ to be detected by Theorem~\ref{thm:main_cons}. In future work \cite{P:Joins} we will also consider products of $3$-dimensional s.d.\ reflexive polytopes with a length two line segment, related to an algebro-geometric version of the suspension of a K$3$ surface; we expect many of the Calabi--Yau threefolds constructed by Lee in~\cite{L17} appear in this way. Among the examples we consider in \S\ref{sec:examples}, we describe pairs $(P,D)$, where $P$ is the product of two lattice hexagons, in particular detail.

\begin{pro}
	\label{pro:prod_dP6}
	Let $P_6$ be the integral hexagon associated with the toric del Pezzo surface of degree $6$. Consider the four-dimensional polytope $P := P_6 \times P_6$. From the toric variety $X_P$ we can construct $14$ topological types of Calabi--Yau threefolds, $5$ of which have $b_2=1$ and invariants which do not appear in the list of Kapustka~\cite{K15}, or among recent constructions of Lee~\cite{L17,L18}.
\end{pro}

In a somewhat different language, the toric variety $X_P$ associated to $P$ has $36$ ordinary double point singularities and $12$ singularities locally isomorphic to the anti-canonical cone on the del Pezzo surface of degree $6$. The latter $12$ singularities admit two deformation components; the archetypal `Tom' and `Jerry' (see Brown--Kerber--Reid~\cite{BKR12}). Famous results of Friedman \cite{F91} and Tian \cite{T92} show that smoothing nodal singularities is generally obstructed. Consonant with this, we must verify a global condition (existence of a polarisation) before we are able to smooth all $48$ singularities of $X_P$. In this language, we expect the join construction of Knapp--Sharpe~\cite[\S$2.5$]{KS19} to correspond to the `simultaneous Jerry' smoothing. We show that Calabi--Yau threefolds recently obtained by Inoue~\cite{I19} via other join constructions also fit into our framework.

Finally, we note that there are geometric transitions between the Calabi--Yau threefolds we construct, and those obtained via Batyrev's construction \cite{B94}. While these transitions are more general than conifold transitions (indeed, the singular locus appearing in the middle of the transition is generally non-isolated) one can view the construction we present as an extension of the approach taken by Batyrev--Kreuzer in \cite{BK10}.

---------------------------------------

\subsection*{Acknowledgements}

We thank Johanna Knapp for bringing our attention to recent progress on smoothing joins of elliptic curves. There is also a clear intellectual debt owed to the work of Gross and Siebert, and to Gross' study of topological mirror symmetry. TP was supported by a Fellowship by Examination at Magdalen College, Oxford.
\section{Simply decomposable polytopes}
\label{sec:sd-poytopes}

We construct smoothings of Calabi--Yau toric hypersurfaces with singularities belonging to a certain class. The following definition provides a combinatorial description of this class of toric singularities. Recall that we say that two lattice polytopes are \emph{equivalent} if they differ by the composition of an integral linear map and a translation by a lattice vector.

\begin{dfn}
	\label{dfn:sd_polytope}
	Given a four-dimensional reflexive polytope $P$, we say that $P$ is \emph{simply decomposable (s.d.\@)} if every $2$-dimensional face of $P$ admits a Minkowski decomposition into lattice polytopes, each of which is equivalent to a standard simplex.
\end{dfn}

We assume throughout this article that $P \subset N_\RR$, where $N_\RR = N \otimes_\ZZ \RR$ and $N \cong \ZZ^4$. We let $M$ denote the lattice dual to $N$, and set $M_\RR := M \otimes_\ZZ \RR$.

\begin{dfn}
	\label{dfn:std_decomp}
	Given a four dimensional s.d.\ reflexive polytope $P$, let $D$ be a function which sends each polygon $Q$ in
	\[
	\left\{\ell(F^\star)F : F \in \Faces(P,2)\right\}.
	\]
	to a Minkowski decomposition of $Q$. If each summand in each element of $D$ is equivalent to a standard simplex, we say that $D$ is a \emph{standard decomposition}. We write $D(\rho)$ the Minkowski decomposition of $\rho \in \Faces(P,2)$ contained in $D$. Note that each component $D(\rho)$ of $D$ is a multiset. 
\end{dfn}

\begin{rem}
	The Minkowski decompositions which appear in $D$ need not be \emph{lattice} Minkowski decompositions, for example we admit the Minkowski decomposition illustrated in Figure~\ref{fig:minksum}.
\end{rem}

\begin{figure}
	\includegraphics[scale=0.25]{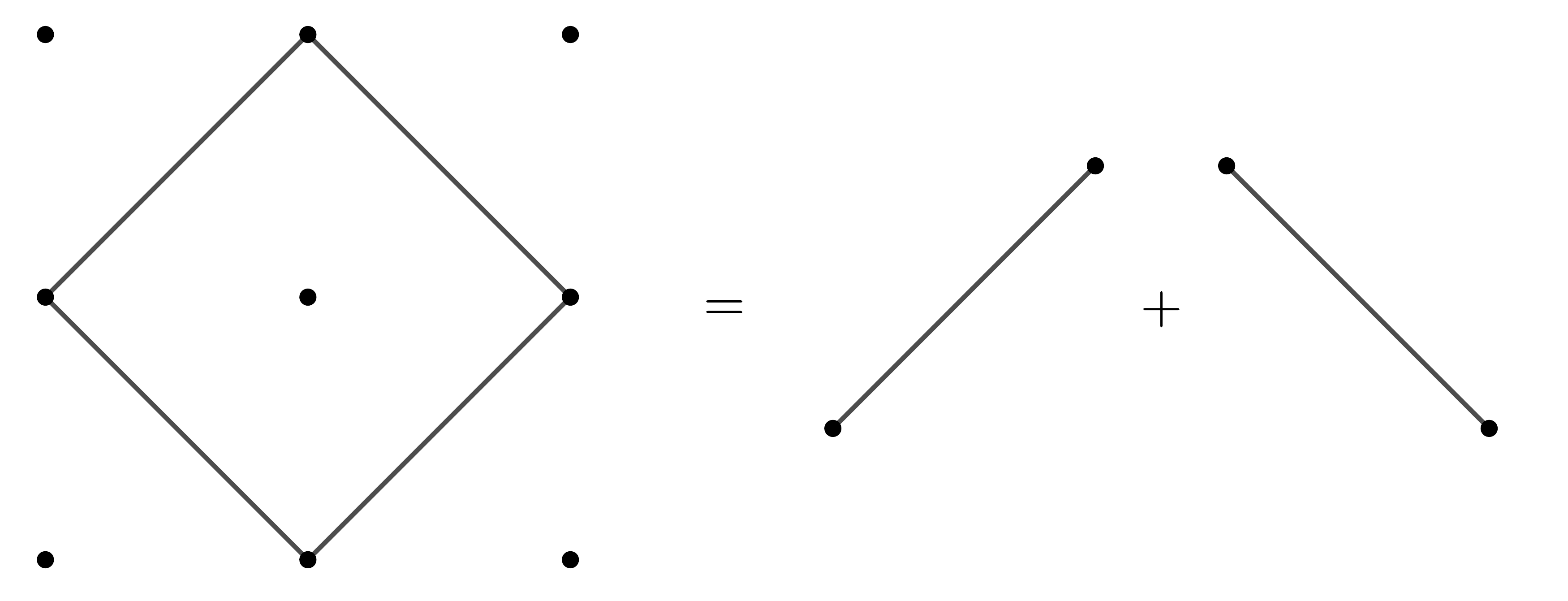}
	\caption{Example of a non-lattice Minkowski decomposition.}
	\label{fig:minksum}
\end{figure}

In \S\ref{sec:smoothing_polytope} we define an integral affine structure on $\partial P$. However, before this can be used as input to the Gross--Siebert algorithm, this integral affine structure must admit a `polarisation'. This is equivalent to the existence of a certain piecewise linear (PL) function on $\partial P$. To describe this we first note that, fixing $\rho \in \Faces(P,2)$, there is a canonical matching (or multi-valued function) between the edges of any $m \in D(\rho)$ and the edges of $\rho$. We let $\Edges(\rho,m) \subseteq \Edges(\rho)$ denote the subset of edges of $\rho$ matched with some edge of $m$. Observe that, given a $2$-dimensional face $\sigma \in \Faces(P^\circ,2)$ and edge $\tau \in \Edges(\sigma)$, the set
\[
S(\sigma,\tau) := \{m \in D(\tau^\star) : \sigma^\star \in \Edges(\tau^\star,m)\} 
\]
is in bijection with lattice length one segments of $\ell(\sigma^\star)\tau$. Moreover, after fixing an orientation of the edges and $2$-dimensional faces of $P^\circ$, and an ordering of $D(\tau^\star)$, there is a canonical such bijection.

\begin{dfn}
	\label{dfn:regular}
	Let $(\mu_\sigma : \sigma \in \Faces(P^\circ,2))$ be a tuple of PL functions on the polygons $\ell(\sigma^\star)\sigma$ which are strictly convex on a maximal triangulation. Let $\sigma_1$ and $\sigma_2$ be $2$-dimensional faces of $P^\circ$ meeting along an edge $\tau \in \Edges(P^\circ)$. Let $a_1$ and $a_2$ be segments of $\ell(\sigma^\star_1)\tau$ and $\ell(\sigma^\star_2)\tau$ which are both identified with a single element of $S(\sigma_1,\tau) \cap S(\sigma_2,\tau)$. We call a tuple of PL functions \emph{admissible} if the slope of $\mu_{\sigma_1}$ along $a_1$ coincides with the slope of $\mu_{\sigma_2}$ along $a_2$ for all $\tau$, $\sigma_i$, and $a_i$ for $i \in \{1,2\}$. If $(P,D)$ admits an admissible tuple of PL functions, we say that $(P,D)$ is \emph{regular}.
\end{dfn}

We present an algorithm to verify whether such a function exists in general, and include Magma code to check this in supplementary material. To describe this algorithm we introduce the notion of a \emph{slope function} $V$ for $(P,D)$. These functions determine the slopes of a PL function along the edges of $P^\circ$.

Fix a polytope $P$ and $D$ a set of Minkowski decompositions of its $2$-dimensional faces. Let $V$ be a function taking each factor $m$ in the multiset 
\[
\bigcup_{\rho \in \Faces(P,2)}{D(\rho)}
\]
to an element in $\QQ$. We call $V$ a slope function for the pair $(P,D)$. Observe that an admissible tuple of PL functions for $(P,D)$ uniquely determines a slope function.

Recall that we have fixed orientations of each $\sigma \in \Faces(P^\circ,2)$ and edge $\tau \in \Edges(\sigma)$. We set $\sgn(\sigma,\tau) := 1$ if the orientation of $\tau^\star$ agrees with the clockwise ordering of the edges of $\sigma$, and $\sgn(\sigma,\tau) := -1$ if not.

\begin{dfn}
	\label{dfn:consistent}
	We call a slope function \emph{consistent} if, for each $\sigma \in \Faces(P^\circ,2)$, we have that
	\[
		\sum_{\tau \in \Edges(\sigma)}\sum_{\substack{m \in D(\tau^\star)\\ \sigma^\star \in \Edges(\tau^\star,m)} }{\sgn(\sigma,\tau)V(m)} = 0.
	\]
\end{dfn}

\begin{dfn}
	\label{dfn:strictly_convex}
	We call a slope function \emph{strictly convex} if elements in the multi-set $\{V(m) : m \in D(\rho)\}$ are pairwise distinct for any $\rho \in \Faces(P,2)$.
\end{dfn}


Fixing a $\sigma \in \Faces(P^\circ,2)$, we define $\bar{\sigma} := \ell(\sigma^\star)\sigma$ and, given an element $\tau \in \Edges(\sigma)$, we define $\bar{\tau} := \ell(\sigma^\star)\tau$. Fix a face $\sigma \in \Faces(P^\circ,2)$ and an edge $\tau \in \Edges(\sigma)$. Given a consistent strictly convex slope function $V$, we fix a piecewise linear function $\psi_\sigma$ on $\partial \bar{\sigma}$ with slope equal to $V(m)$ under the bijection between length one segments of $\bar{\tau}$ and $S(\sigma,\tau)$. Note that, appropriately ordering $D(\tau)$, $\{V(m) : m \in D(\tau) \}$ forms a monotone sequence, and $\psi_\sigma$ is convex on each $\bar{\tau}$. This piecewise linear function determines a polyhedral decomposition $T_0(\bar{\sigma})$ of $\bar{\sigma}$ by projecting $2$-dimensional faces of the convex hull $\Gamma$ of the points
\[
\{(x,\psi_\sigma(x)) : x \in \partial(\bar{\sigma})\}.
\] 

Following common terminology in polyhedral combinatorics, see for example \cite{NZ11}, we refer to a polytope such that only lattice points are its vertices as an \emph{empty} polytope. Similarly, we refer to a polytope whose set of lattice points equals the set of lattice points contained in its boundary as a \emph{hollow} polytope.

\begin{dfn}
	\label{dfn:regular_slope}
	 We call a consistent strictly convex slope function $V$ on $(P,D)$ \emph{regular} if, for any $\sigma \in \Faces(P^\circ,2)$ and any empty polygon $Q$ in $T_0(\bar{\sigma})$ is equivalent to a standard triangle. 
\end{dfn}

\begin{pro}
	\label{pro:also_regular}
	The pair $(P,D)$ is \emph{regular} if and only if there exists a consistent strictly convex slope function $V$ on $(P,D)$ such that for any $\sigma \in \Faces(P^\circ,2)$, and any polygon $Q$ in $T_0(\sigma)$, if $Q$ contains no lattice points other than its vertices, $Q$ is equivalent to a standard triangle. 
\end{pro}
\begin{proof}
	Fix a face $\sigma \in \Faces(P^\circ,2)$. Let $\Gamma$ be, as above, the graph of the PL function on $\bar{\sigma}$ determined by $V$. Note that that, since $V$ is strictly convex, $\psi_\sigma$ is strictly convex on each edge $\bar{\tau}$ of $\bar{\sigma}$ (that is, the function bends non-trivially at each integral point on each edge).

	We iteratively modify $\Gamma$ to define a convex PL function $\psi_\sigma$ on $\bar{\sigma}$. In particular, fix a lattice point $x$ in the relative interior of $\bar{\sigma}$, and let $h(x) \in \QQ$ be such that $(x,h(x)) \in \Gamma$. We replace $\Gamma$ with the convex hull of its vertex set together with $(x,h(x)-\eta)$ for a sufficiently small value of $\eta \in \QQ_{> 0}$. This modification induces a star subdivision of the triangulation induced by $\Gamma$ at $x$. Iterating this modification over every integral $x$ in the relative interior of $\bar{\sigma}$, we obtain a piecewise linear function which induces a polyhedral decomposition $T$ of $\bar{\sigma}$. It remains to check that $T$ corresponds to a maximal triangulation of $\bar{\sigma}$. We consider two cases:
	\begin{enumerate}
		\item By construction, every polygon meeting an integral point $x$ in the relative interior of $\bar{\sigma}$, is a triangle; moreover, since every such $x$ is a vertex, all such triangles are empty, and hence standard.
		\item Since $V$ is regular and strictly convex, polygons which do not contain any such $x$ are also standard.
	\end{enumerate}
	
	Conversely, given an admissible tuple of PL functions, the slopes along the edges define a consistent strictly convex slope function $V$. If this were not regular, $\mu_\sigma$ would not be admissible for some $\sigma \in \Faces(P,2)$. Indeed, the triangulation induced by the restriction of $\mu_\sigma$ to $\partial \bar{\sigma}$ (analogous to the construction of the `unmodified' graph $\Gamma$) contains an empty polygon which is not a standard simplex. This polygon is a domain of linearity for $\mu_\sigma$, contradicting admissibility.
\end{proof}

We note that the condition that empty polygons in each triangulation $T_0(\bar{\sigma})$ are standard triangles is a generic condition, and we completely classify the situations in which it can fail.

\begin{pro}
	\label{pro:irregular}
	Let $T$ be a polyhedral decomposition of a polygon $\bar{\sigma}$ such that the vertex set of $T$ is equal to the set of integral points in $\partial \bar{\sigma}$. If $T$ contains an empty polygon $Q$, and $Q$ is \emph{not} isomorphic to a standard simplex, then $Q$ is equivalent to an empty lattice square and $\bar{\sigma}$ is hollow.
\end{pro}
\begin{proof}
	It is well known that the only empty polygons are those equivalent to the standard simplex, and empty lattice square. Let $Q$ is an empty lattice square in $T$ and, without loss of generality assume that $Q$ is the convex hull of the points given by the columns of the matrix
	\[
	\begin{matrix}
	0 & 1 & 1 & 0 \\
	0 & 0 & 1 & 1.
	\end{matrix}
	\]
	Suppose there exists a point $(a,b) \in Q$ such that $a,b < 0$; and thus that $(0,0)$ lies in the convex hull of $x$ and the points $(0,1)$ and $(1,0)$. However this is a contradiction to $(0,0)$ lying in the boundary of $Q$. Thus $Q$ is contained in the union of $[0,1] \times \RR$ and $\RR \times [0,1]$; the only such polygons are hollow.
\end{proof}

It is well known that the only hollow polygons are either equivalent to the Cayley product of two line segments or to double the standard simplex. Thus the only possible cases in which regularity may fail are:
\begin{enumerate}
	\item If $\bar{\sigma}$ is a Cayley product of two line segments, $T(\bar{\sigma})$ contains an empty square face if and only if opposite length one segments have the same slope.
	\item  If $\bar{\sigma}$ is a twice the standard simplex, $T(\bar{\sigma})$ contains an empty square face if and only if, up to an affine linear transformation $\bar{\sigma}$ can be taken to the configuration shown in Figure~\ref{fig:bad_triangle}; where $a$, $b$ and $c$ denote the slopes along the edges they label and $a+b+c=0$.
\end{enumerate}

\begin{figure}
	\includegraphics[scale=0.25]{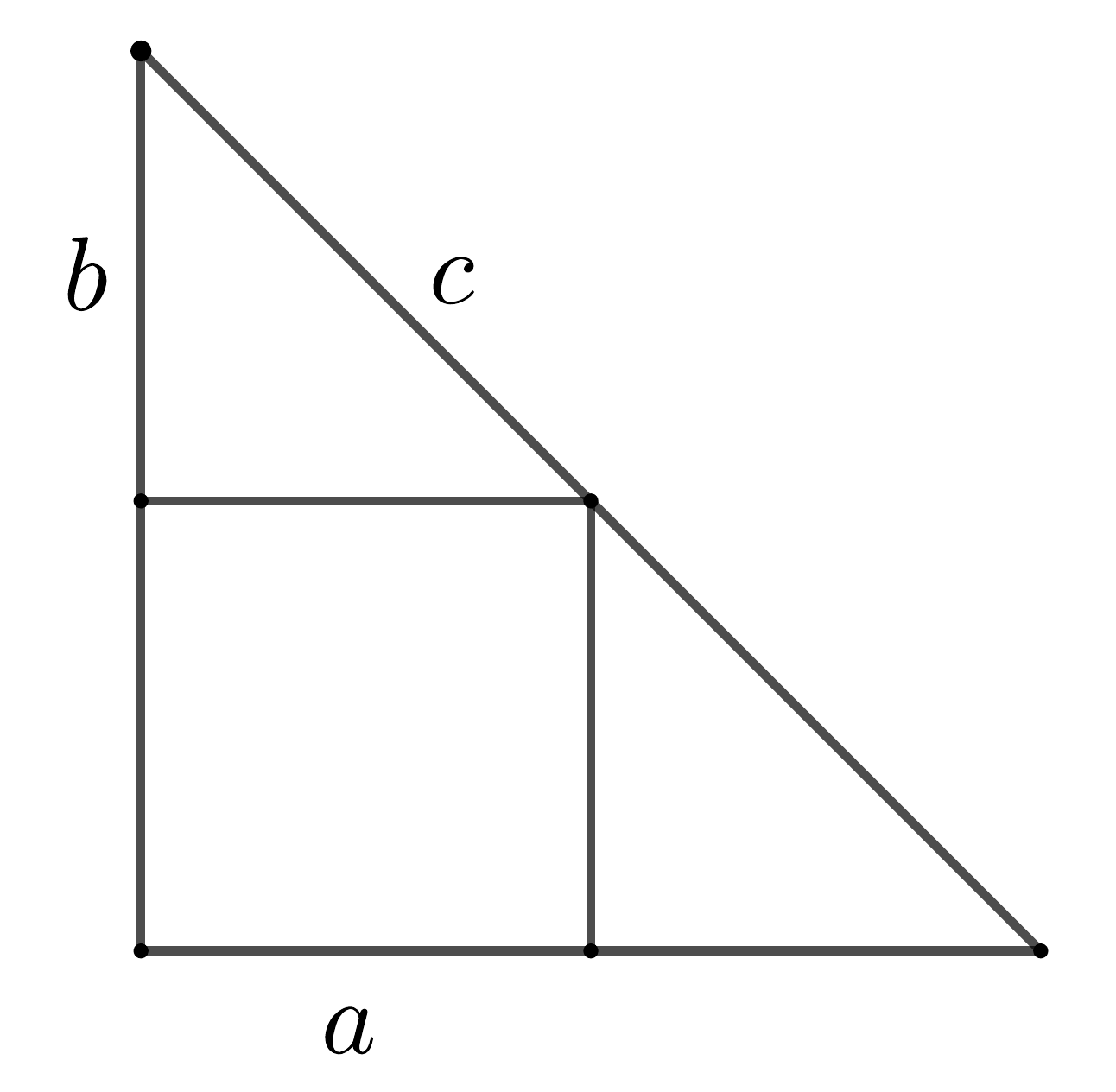}
	\caption{Irregular strictly convex slope function.}
	\label{fig:bad_triangle}
\end{figure}



\section{Smoothing the boundary of a reflexive polytope}
\label{sec:smoothing_polytope}
\subsection{A degenerate affine structure}
\label{sec:degenerate}

The Gross--Siebert algorithm requires both \emph{discrete} and \emph{algebraic} input. The discrete data consists of a triple $(B,\P,\varphi)$ where $B$ is an integral affine manifold; $\P$ is a polyhedral decomposition; and $\varphi$ is a multi-valued piecewise linear function, see \cite[Definition~$1.45$]{GS06}. In this section we show how to assign such discrete data to a $4$-dimensional s.d.\ polytopes, together with choices of Minkowski decompositions of its two-dimensional faces. We also describe the algebraic input in \S\ref{sec:main_result}; in fact, using the main results of \cite{GS06}, we explicitly describe the space of admissible algebraic data.

Given a reflexive polytope $P$ there is a construction (using the degeneration of a general anti-canonical section of $X_P$ to the union of all codimension $1$ toric strata of $X_P$), following Gross~\cite[\S$2$]{G05}, of an integral affine structure on $\partial P$ without any further input data. We follow the description given in \cite[Definition~$2.10$]{G05}.

\begin{cons}
	Let $\Delta^c$ denote the union of ($1$-dimensional) cells of the first barycentric subdivision of $\partial P^\circ$ which do not intersect the interior of a facet of $P^\circ$, or contain a vertex of $P^\circ$. We define an integral affine manifold with singularities $B^c$ by defining an integral affine structure on $B^c \setminus \Delta^c$ as follows.
	\begin{enumerate}
		\item For each facet $F$ of $P^\circ$, the affine structure on the interior of $F$ is determined by the composition $F \hookrightarrow P^\circ \hookrightarrow \RR^4$. Identifying the minimal affine linear space containing $F$ with $\RR^3$ identifies the interior of $F$ with a domain in $\RR^3$.
		\item Given a vertex $v$ of $P^\circ$, let $W_v$ denote the open star of $v$ in the first barycentric subdivision of $\partial P^\circ$. We define the affine chart 
		\[
		\psi_v \colon W_v \to M_\RR / \langle v \rangle 
		\]
		by projection.
		\end{enumerate}  
\end{cons}

The singular locus of this affine structure is a graph, with vertices located at the barycentres of the $2$-dimensional faces and edges of $P^\circ$. Loops around the edges of $\Delta$ define \emph{monodromy} for this integral affine structure. This monodromy was computed in this case in \cite[Proposition~$2.13$]{G05}, similar to calculations made by Ruan \cite{R07} and Haase--Zarkov~\cite{HZ05}.

Fix vertices $v$ and $v'$ of $P^\circ$ contained in facets $F_1$ and $F_2$. Choose a loop $\gamma$ based at $v_1$ which passes successively into the interior of $F_1$, though $v_2$, the interior of $F_2$, and back to $v_1$. The monodromy of the affine structure around $\gamma$ is given by the linear map
\[
T_\gamma(m) = m + \langle n_2-n_1, m \rangle (v_1-v_2).
\] 
where $n_1$ and $n_2$ are the lattice points at vertices dual to the facets $F_1$ and $F_2$ respectively. In particular, assuming that $v_1$ and $v_2$ are contained in an edge $\tau$ of a two-dimensional face $\sigma = F_1 \cap F_2$, in suitable co-ordinates $T_\gamma(m)$ is given by the matrix
\[
\begin{pmatrix}
1 & 0 & \ell(\sigma^\star)\ell(\tau) \\
0 & 1 & 0 \\
0 & 0 & 1
\end{pmatrix}.
\]

In \cite{GS06, Gross--Siebert} the authors show how to form a toric log Calabi--Yau space from an integral affine manifold with \emph{simple} singularities (and additional data). The general definition of simplicity for the singular locus of an affine structure is given in \cite[Definition~$1.60$]{GS06}. The restriction to the case $\dim B =3$ is described in detail in \cite[Example~$1.62$]{GS06} and we recall the main points of this description. We first fix an integral affine manifold $B$ and polyhedral decomposition $\P$; recalling that polyhedral decompositions of integral affine manifolds with singularities are carefully defined in \cite[Definition~$1.22$]{GS06}.

\begin{enumerate}
	\item The discriminant locus $\Delta$ is a trivalent graph, with vertices contained a minimal cell $\tau$ of dimension $1$ or $2$.
	\item Loops passing singly around a segment of the discriminant locus induce monodromy operators on the tangent spaces of $B$. The corresponding matrices are conjugate to 
	\[
	T_g := 
	\begin{pmatrix}
	1 & 0 & 1 \\
	0 & 1 & 0 \\
	0 & 0 & 1
	\end{pmatrix}.
	\]
	\item\label{it:positive} If $v$ is a vertex of $\Delta$ contained in an edge, the monodromy matrices induced by loops passing singly around edges of $\Delta$ containing $v$ are simultaneously congruent to the matrices
	\begin{align}
		\label{eq:positive_monodromy}
		T_1 := \begin{pmatrix}
		1 & 1 & 0 \\
		0 & 1 & 0 \\
		0 & 0 & 1
	\end{pmatrix} &&
	T_2 := \begin{pmatrix}
		1 & 0 & 1 \\
		0 & 1 & 0 \\
		0 & 0 & 1
	\end{pmatrix} &&
	T_3 := \begin{pmatrix}
		1 & -1 & -1 \\
		0 & 1 & 0 \\
		0 & 0 & 1
	\end{pmatrix}.
	\end{align} 
	\item If $v$ is a vertex of $\Delta$ not contained in an edge, the monodromy matrices induced by loops passing singly around edges of $\Delta$ containing $v$ are simultaneously congruent to the inverse transposes of the matrices appearing in \eqref{it:positive}.
\end{enumerate}

Fixing a reflexive polytope $P$, the integral affine manifold $B^c$ does not have simple singularities in general, for any choice of polyhedral decomposition; indeed $\Delta^c$ is not usually trivalent, and the monodromy operators around segments of $\Delta^c$ are not usually of the correct form. In \cite{G05,HZ05} the authors construct an affine manifold related to $B^c$ from triangulations of both $\partial P$ and $\partial P^\circ$. We describe an alternative perturbation of $\Delta^c$, in the case that $P$ is an s.d.\ polytope.

\subsection{Constructing a polarisation}
\label{sec:cons_polarisaton}

We construct a convex piecewise linear function $\varphi^r$ on $\partial P^\circ$, and use its domains of linearity to define a polyhedral decomposition $\P$. Later, we fix \emph{fan structures}, see \cite[Definition~$1.1$]{Gross--Siebert}, at each vertex of this decomposition to define an integral affine structure. The function $\varphi^r$  (the superscript $r$ denotes \emph{refined}, in contrast with the superscript $c$ in $B^c$ which denotes \emph{coarse}) determines a polarisation $\varphi$ on this tropical manifold.

\begin{cons}
	\label{cons:polarisation}
	Fix an s.d.\ polytope $P$ and a standard decomposition $D$ of its $2$-dimensional faces. Let $\varphi^c$ be the  piecewise linear function on $M_\RR$ which evaluates to $1$ at each vertex of $P^\circ$ (and corresponds to an anti-canonical divisor of the toric variety defined by the normal fan of $P$). We define
	\[
	\varphi^r := \varphi^c + \epsilon \psi
	\] 
	where $\epsilon \in \QQ_{>0}$ is a sufficiently small rational value (the precise value of $\epsilon$ is unimportant as the domains of linearity of $\varphi^r$ are constant for all sufficiently small values). The function $\psi$ is a piecewise linear (not necessarily convex) function which we now construct.
	
	In order to construct $\psi$ we assume that $(P,D)$ is regular, and fix an admissible tuple of PL functions $(\mu_\sigma : \sigma \in \Faces(P,2))$. Let $V$ denote the corresponding strictly convex slope function. Note that in choosing such a function we implicitly fix an orientation on the edges and $2$-dimensional faces of $P^\circ$; as well as an ordering of $D(\rho)$ for each $\rho \in \Faces(P,2)$ which induces a monotone ordering of $\{V(m) : m \in D(\rho)\}$ . 
	
	Given an edge $\tau$ of $P^\circ$, subdivide $\tau$ into $|D(\tau^\star)|+2$ intervals. The ordering of $D(\tau^\star)$ determines a bijection between the $|D(\tau^\star)|$ intervals which do not meet a vertex of $\tau$, and $D(\tau^\star)$. Let $\psi|_\tau$ be linear on each segment with slope equal to $V(m)$ along the segment corresponding to $m \in D(\tau^\star)$. We insist that $\psi(v) = 0$ at each vertex $v$ of $\tau$; in particular $\psi|_\tau$ is a non-positive function. These conditions uniquely determine $\psi_\tau$ up to a single real parameter (the value of $\psi(x)$ for any $x$ in the relative interior of $\tau$). We fix this parameter arbitrarily.
	
	Next, fix a $2$-dimensional face $\sigma \in \Faces(P^\circ,2)$. To define $\psi|_\sigma$ we choose an embedding $j$, a composition of a rational scaling and translation, of $\bar{\sigma} = \ell(\sigma^\star)\sigma$ into the relative interior of $\sigma$. We define $\sigma' := j(\bar{\sigma})$. Recall that, fixing an edge $\tau$ of $\sigma$, the set 
	\[
		\{m \in D(\tau^\star) : \sigma^\star \in \Edges(\tau^\star,m)\} 
	\]
	is in canonical bijection with the set of lattice one line segments of $\bar{\tau} = \ell(\sigma^\star)\tau$. For a point $x \in \sigma'$ we let $\psi(x)$ be $\mu_\sigma(j^{-1}(x))- K$, for a fixed large positive integer $K$. Note that we have also defined $\psi$ over $\partial \sigma$, as this is a union of edges. We extend $\psi$ between $\partial \sigma$ and $\partial \sigma'$ by taking $\psi|_\sigma$ to be the unique PL function with graph equal to the convex hull of the points $(x,\psi(x))$ where $x \in \sigma' \cup \partial \sigma$. We illustrate an examples of the resulting triangulation in Figure~\ref{fig:triangulation}, in which $\sigma'$ is shaded.

	\begin{figure}
		\includegraphics[scale=0.4]{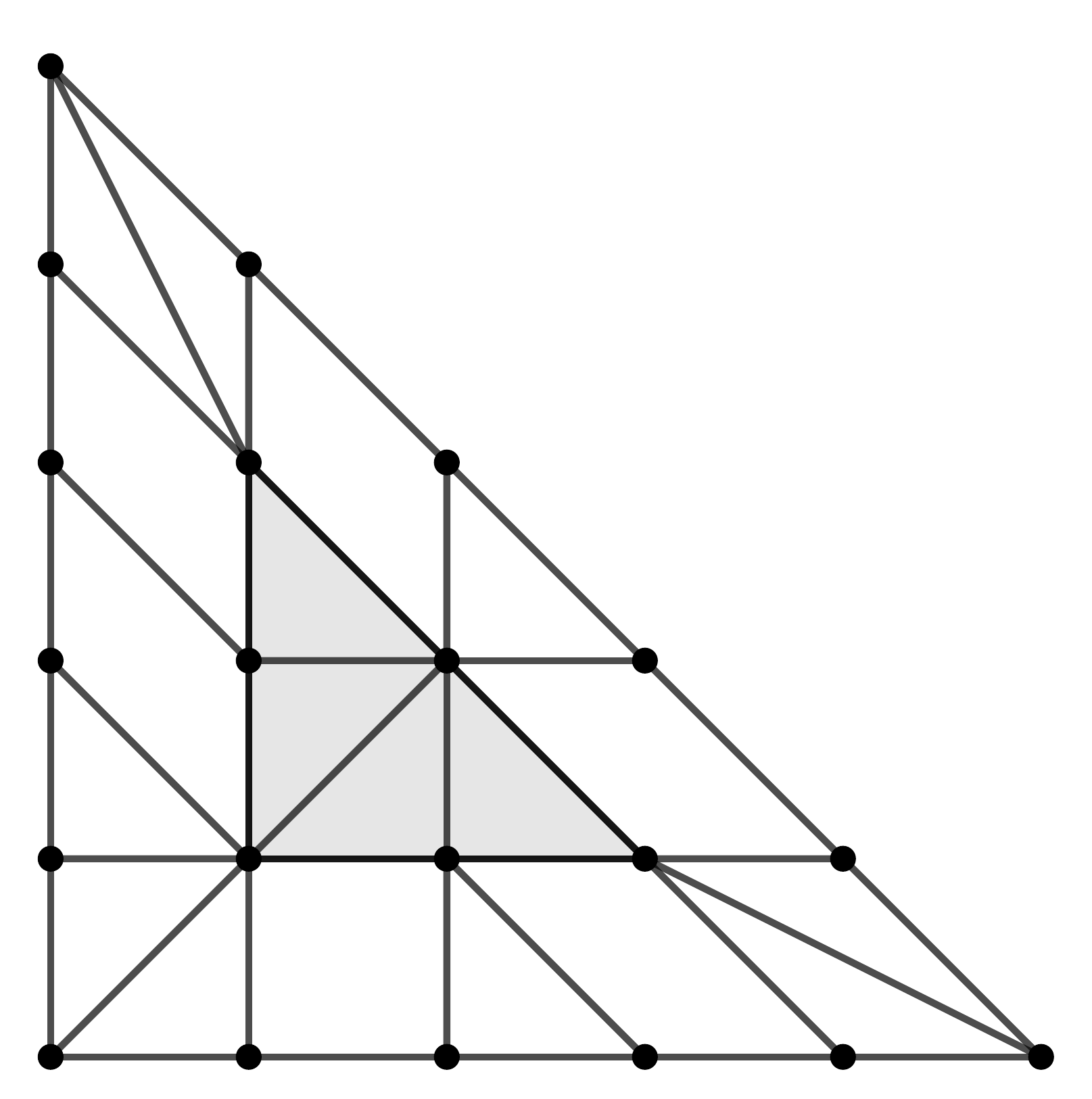}
		\caption{The triangulation of a face $\sigma$ induced by $\psi_\sigma$.}
		\label{fig:triangulation}
	\end{figure}

	Finally, fix a facet $\rho$ of $P^\circ$. We extend $\psi$ arbitrarily across $\rho$ subject to the condition that the domains of linearity are either simplices, or polyhedral cones over rectangular faces in a $2$-dimensional face, and $\psi|_\rho$ is convex. Such a triangulation can be achieved by star subdivision of $\rho$, and repeated star subdivisions of the $3$-dimensional cells of this decomposition preserve this condition.
\end{cons}

Via Construction~\ref{cons:polarisation} we have a convex piecewise linear function $\varphi^r$ on $P^\circ$, and we may define a polyhedral decomposition $\P$ of $P^\circ$ via the domains of linearity of $\varphi^r$.

\subsection{Tropical manifolds via fan structures}
\label{sec:tropical_manifold}

We now define an integral affine structure $B$, such that the pair $(B,\P)$ has simple singularities along its discriminant locus $\Delta$. Following the description of such structures used in \cite{Gross--Siebert} we define the integral affine structure $B$ via \emph{fan structures} \cite[Definition~$1.2$]{Gross--Siebert}. Recall that a fan structure consists of maps from the open stars of faces of $\P$ to $\RR$, together with certain compatibility conditions. We define fan structures at the vertices of $P^\circ$, and inductively extend them over the vertices of $\P$. In what follows we let $F_\tau$ denote the minimal face of $P^\circ$ which contains a face $\tau$ of $\P$.

\begin{cons}
	Inductively, we assume that we have defined a fan structure at a vertex $v$ of $\P$, and that $\tau \in \Edges(\P)$ has vertices $\{v,v'\} \in \V{\P}$. Let $\bar{v} \in \V{P^\circ}$ be a vertex of $F_\tau$. We also assume inductively that the fan structure $\pi_v$ on the open star of $v \in \V{\P}$ is given by a composition $T \circ \pi_{\bar{v}}$ for a piecewise linear map $T$ (defined on $\RR^3$) and that $F_v \subseteq F_{v'}$.
	
	First, if $v = v' = \bar{v} \in \V{P^\circ}$, we define a fan structure via the projection $\pi_v \colon M_\RR \to M_\RR/\langle v \rangle \cong \RR^3$. Cones in this fan structure are given by the images of cells in $\P$. We now inductively extend the fan structure to other vertices of $\P$.
	\begin{enumerate}
		\item If $\tau \subset \sigma'$ for some $\sigma \in \Faces(P^\circ,2)$, $v = j(v_1)$ and $v' = j(v_2)$ for some vertices $v_1$ and $v_2$ of $\bar{\sigma}$. We set $\pi_{v'} = T \circ \pi_v$ where
		\[
		T(x) := x + \min\{0,\langle x,u \rangle\}\pi_v(v_1-v_2)
		\]
		\item If $\tau \subset \rho$ for some $\rho \in \Edges(P^\circ)$, and $\tau$ is a segment of $\rho$ which corresponds to a summand $m \in D(\rho^\star)$, we set $\pi_{v'} = T \circ \pi_v$ where
		\[
		T(x) := x + \min_{w \in \V{m}}\{\langle x,w \rangle\}\pi_v(d_\rho)
		\]
		where $d_\rho$ is the primitive direction vector along $\rho$ pointing from $v$ to $v'$.
		\item In any other case, the fan structure at $v'$ is given by applying $\pi_v$ to the open star of $v'$.
	\end{enumerate}
\end{cons}

In all cases the cones of the fan structure are given by the images of cells in $\P$. Note that there are various sign choices involved: the choice of orientation of $2$-dimensional faces and edges; however, different choices yield equivalent choices of fan structure (choices which differ by a linear function). We let $B(P,D)$ denote the integral affine manifold with singularities defined by these fan structures.

As defined, the discriminant locus of $B(P,D)$ is not a trivalent graph. However, following \cite[Proposition~$1.27$]{GS06} (see also \cite[\S$4$]{G05}) we may extend the affine structure over any branch of the discriminant locus around which monodromy of the lattice $\Lambda$ is trivial. The following result is now immediate from standard computations of affine monodromies. We illustrate an example of the discriminant locus $\Delta$ in a face $\sigma \in \Faces(P^\circ,2)$ in Figure~\ref{fig:delta}.

\begin{pro}
	Extending the affine structure of $B(P,D)$ over the complement of the minimal discriminant locus defines an affine manifold with simple singularities.
\end{pro}

\begin{figure}
	\includegraphics[scale=0.4]{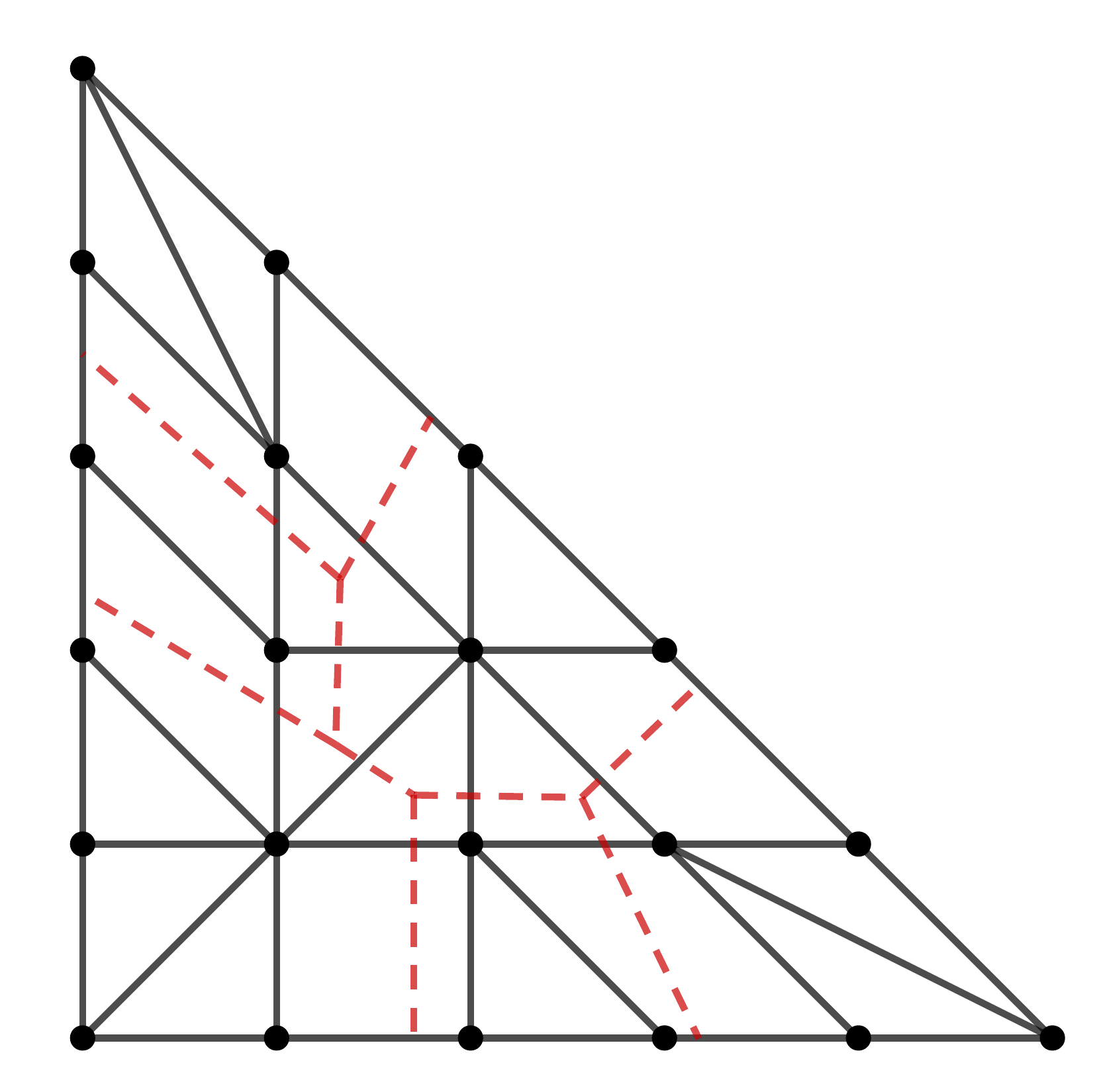}
	\caption{The discriminant locus in a $2$-dimensional face.}
	\label{fig:delta}
\end{figure}

The function $\varphi^r$ provides local representatives for a strictly convex multi-valued piecewise linear function $\varphi$ on $B$. Indeed, this follows immediately from the linearity of $\varphi^r$ on each cell of $\P$; see for example \cite[Definition~$2.14$]{G05}.

We next explain how to incorporate the algebraic data required to apply the Gross--Siebert reconstruction algorithm \cite[Theorem~$1.30$]{Gross--Siebert} into this construction, from which we deduce the existence of a Calabi--Yau threefold $X$ obtained as a smoothing.

\subsection{Proof of Main Result}
\label{sec:main_result}

We recall the statement of our main result concerning the construction of smooth Calabi--Yau threefolds from toric hypersurfaces.

\begin{thm}
	\label{thm:main_cons_2}
	Let $(P,D)$ be a pair consisting of a $4$-dimensional s.d.\ polytope $P$, and a standard decomposition $D$. The pair $(P,D)$ determines a locally rigid, positive, toric log Calabi--Yau space $X_0(P,D)$. Moreover, the toric log Calabi--Yau space $X_0(P,D)$ admits a polarisation if $(P,D)$ is regular.
\end{thm}

This result now follows directly from the results and constructions of \cite{Gross--Siebert}, applied to the polarised tropical manifold $(B,\P,\varphi)$.

\begin{proof}[Proof of Theorem~\ref{thm:main_cons}]
	As explained in \S\ref{sec:cons_polarisaton} and \S\ref{sec:tropical_manifold}, we can construct a polarised tropical manifold $(B,\P,\varphi)$ from the pair $(P,D)$. Following \cite[\S$1.2$]{Gross--Siebert} we can form the underlying variety of a toric log CY space $X_0(B,\P,s)$ once we have fixed \emph{open gluing data} $s$. It suffices for our purposes to suppress the choice of open gluing data; in the terminology of \cite{Gross--Siebert} we choose open gluing data cohomologous to zero (sometimes called \emph{vanilla} gluing data).

	We now need to fix a \emph{log structure} on $X_0(B,\P,s)$. These are determined by sections of a sheaf
	\[
	\LS \cong \bigoplus_{\rho \in \Faces(\P,2)}{\cN_\rho},
	\]
	where the sheaves $\cN_\rho$ are certain rank one sheaves determined by $(B,\P,s)$, such that 
	\begin{enumerate}
		\item the section contains no toric stratum of $X_0(B,\P,s)$ and
		\item the section satisfies the compatibility condition \cite[$(1.8)$]{Gross--Siebert}.
	\end{enumerate}
	Locally, each $2$-dimensional toric stratum $X_\rho$ of $X_0(B,\P,s)$ is either isomorphic to $\PP^2$ (and corresponds to a triangle in a maximal triangulation of $\bar{\sigma}$ for some $\sigma \in \Faces(P^\circ,2)$), in which case $\cN_\rho$ is $\cO_{\PP^2}(1)$; or $X_\rho$ has a map to $\PP^1$, in which case $\cN_\rho$ is the sheaf associated to a fibre of this map.
	
	Since $B$ is positive and simple, we can apply the main results of \cite{GS06}. Indeed, by \cite[Theorem~$5.2$]{GS06}, there is a unique normalised section of the bundle $\LS$ which induces a log CY structure on $X_0(B,\P,s)$. Moreover, by \cite[Theorem~$5.4$]{GS06}, sections of $\LS$ which determine a toric log CY space are in bijection with $H^1(B,\iota_\star \breve{\Lambda} \otimes_\ZZ \kk^\star)$. We compute this dimension using topological methods in \S\ref{sec:topology}.

	Finally, we need to verify that this log structure is locally rigid, \cite[Definition~$1.26$]{Gross--Siebert}. However, by \cite[Remark~$1.29$]{Gross--Siebert} this follows immediately from the simplicity of the pair $(B,\P)$. Hence we may apply \cite[Theorem~$1.30$]{Gross--Siebert} and \cite[Theorem~$4.4$]{RS:draft} to obtain a family which smooths $X_0(B,\P,s)$. It follows from \cite[Proposition~$2.2$]{GS10} (see also \cite[p.$44$]{GS10}) that the general fibre of such an family has at worst codimension $4$ singularities, and is thus smooth if $X_0(B,\P,s)$ is $3$-dimensional.
\end{proof}

\begin{rem}
	We note that while $X_0(P,D)$ may not be projective, the obstruction to projectivity described in \cite[Theorem~$2.34$]{GS06} vanishes if the gluing data $s$ is a coboundary. Hence, as the diffeomorphism type of $X_0(P,D)$ is not affected by $s$, $h^2(X_0(P,D),\QQ) \geq 1$; a fact we make use of in \S\ref{sec:topology}.
\end{rem}
\section{Topological analysis}
\label{sec:topology}

Given an s.d.\ polytope $P$ and a standard decomposition $D$, we have shown how to construct a Gross--Siebert smoothing from $(P,D)$. We now make an analysis of a topological space $X := X(P,D)$ we can associate with $B := B(P,D)$. The space $X$ admits the structure of a \emph{well-behaved} (see \cite{G01}) torus fibration. In particular, there is a map 
\[
f \colon X \to B,
\]
such that fibres over the smooth locus $B_0 := B \setminus \Delta$ are three dimensional (real) tori. We set $X_0 := X_0(P,D)$ and write $X(B)$ for the topological compactification of the torus fibration associated to any integral affine manifold $B$ with simple singularities.

In the spirit of Conjecture~\ref{conj:KN_spaces}, given a toric log Calabi--Yau space $X_0(P,D)$, we may form the associated \emph{Kato--Nakayama} space $X^{\KN}$, which is known to be a topological model of the general fibre of the Gross--Siebert smoothing. We can attempt to compare this space with the topological space $X$. Progress in this direction has been recently made by Arg\"uz--Siebert in \cite{AS16}.

\subsection{Topological semi-stable torus fibrations}

Fix an integral affine manifold $B$ with simple singularities along $\Delta$. There is a natural torus fibration $T^\star B_0/\breve{\Lambda} \to B_0$, where $\breve{\Lambda}$ is the lattice of integral covectors. We compactify this torus fibration using local models defined by Gross in \cite{G01Ex,G01}, and similar work of Ruan~\cite{R07} (where such fibrations are called type I, II, and III). Summaries of these compactifications can also be found in \cite[Chapter~$6$]{DBranes09} and \cite[\S$2$]{CBM09}. We give the statement of a theorem summarising the output of this compactification and give a brief topological description of each fibre.

\begin{thm}[{\!\!\cite[Theorem~$2.1$]{G01}}]
	\label{thm:compactification}
	Let $B$ be a $3$-manifold and let $B_0 \subset B$ be a dense open set such
	that $\Delta := B \setminus B_0$ is a trivalent graph. Assume that the set of vertices of $\Delta$ are partitioned into sets $\Delta_+$ of \emph{positive} and $\Delta_-$ of \emph{negative} vertices. Suppose there is a $T^3$ bundle $f_0 \colon X(B_0) \to B_0$ such that the local monodromy of $f_0$ is generated, in a suitable basis, by:
	\begin{enumerate}
		\item The matrix $T_g$, as defined in \S\ref{sec:smoothing_polytope}, when $x \in \Delta$ is not a trivalent point.
		\item The matrices $T_i$ for $i \in \{1,2,3\}$ appearing in \eqref{eq:positive_monodromy} when $x \in \Delta_+$.
		\item The inverse transposes of the matrices $T_i$, for $i \in \{1,2,3\}$, when $x \in \Delta_-$.
	\end{enumerate}
	There is a $T^3$ fibration $f \colon X \to B$ and a commutative diagram
	\[
	\xymatrix{
	X(B_0) \ar[d] \ar@{^{(}->}[r] & X \ar[d] \\
	B_0 \ar@{^{(}->}[r] & B.	
	}
	\]
	Over connected components of $\Delta \setminus (\Delta_+ \cup \Delta_-)$, $(X, f, B)$ is conjugate to the generic singular fibration, over points of $\Delta_+$ it is conjugate to the positive fibration and over points of $\Delta_-$ to the negative fibration.
\end{thm}

\begin{rem}
	\label{rem:topology_of_fibres}
	We describe the topology of each type of singular fibre of $f \colon X(B) \to B$.
	\begin{enumerate}
		\item Given a point $b \in \Delta$ which is not trivalent, the fibre of $f \colon X(B) \to B$ over $b$ is equal to the product of a pinched torus with $S^1$.
		\item Given a point in $b \in \Delta_+$, $f^{-1}(b)$ is homeomorphic to a $T^3$ in which a $T^2$ has been contracted. Equivalently, following \cite{R07}, $f^{-1}(b)$ is the suspension of $T^2$ with the two poles identified.
		\item Given a point $b \in \Delta_-$, and writing $T^2$ as a quotient of $[0,1]\times [0,1]$, $f^{-1}(b)$ admits a map to $T^2$, defining a circle bundle over $(0,1)\times (0,1)$ and a one-to-one map over boundary points.
	\end{enumerate}
\end{rem}

Following \cite{CBM09}, we call these \emph{topological semi-stable compactifications} of torus fibrations. The following well-known observation, while straightforward, is fundamental to topological calculations on $X$.
 
\begin{lem}
	\label{lem:euler_number}
	Given a topological semi-stable torus fibration $f \colon X(B) \to B$ the Euler characteristic of $X(B)$ is equal to the difference between the number of positive and negative vertices.
\end{lem}
\begin{proof}
	It follows from the topological descriptions of the fibres given in Remark~\ref{rem:topology_of_fibres} that the only fibres which do not contain a circle factor are the fibres over positive and negative vertices; these fibres have Euler characteristic $+1$ and $-1$ respectively.
\end{proof}

We now compute a number topological invariants of $X(B)$. To this end we recall some additional results from \cite[\S$2$]{G01}.

\begin{pro}[{\!\!\cite[Proposition~$2.13$]{G01}}]
	Given $f \colon X(B) \to B$ as in Theorem~\ref{thm:compactification} and $B$ homeomorphic to $S^3$, then the second Steifel--Whitney class of $X$ vanishes. That is, $X(B)$ is a \emph{spin} $6$-manifold.
\end{pro}

We will make considerable use of the critical locus $\Crit(f)$ which, following \cite[Definition~$2.14$]{G01} is a union of (real) surfaces meeting in finite sets of points. Over non-trivalent points of $\Delta$, the critical locus $\Crit(f)$ restricts to a circle. This circle degenerates to a point over a positive vertex; and degenerates to the wedge union of two circles over a negative vertex.

\begin{pro}[{\!\!\cite[Proposition~$2.17$($2$)]{G01}}]
	Given $f \colon X(B) \to B$ as in Theorem~\ref{thm:compactification} and $B$ homeomorphic to $S^3$, then the first Pontryagin class $p_1(X(B)) = -2\Crit(f) \in H^4(X(B),\QQ)$.
\end{pro}

We also recall the following central result of \cite{G01}, which -- following the remarks made in the proof of Theorem~\ref{thm:main_cons} -- relates the space of log structures with the cohomology of $X$.

\begin{pro}[{see \cite[Lemma~$2.4$]{G01I}}]
	\label{pro:leray_degenerates}
	If $B$ is positive, simple, and simply connected the Leray spectral sequence for $f \colon X(B) \to B$ degenerates at the $E_2$ page. Consequently we have that 
	\begin{align*}
	b_2(X(B)) &= h^1(B,\iota_\star\Lambda \otimes_\ZZ \QQ)\\
	b_3(X(B)) &= 2h^1(B,\iota_\star\breve{\Lambda} \otimes_\ZZ \QQ)+2,
	\end{align*}
	where $\iota$ denotes the inclusion of $B_0$ into $B$.
\end{pro}
\begin{proof}
	Note that, as $B$ is simple $\iota_\star \Lambda = R^1f_\star\ZZ$ and $\iota_\star \breve{\Lambda} = R^2f_\star\ZZ$. The result now follows from \cite[Lemma~$2.4$]{G01I} once we know that $X(B)$ is simply connected, and admits a simply connected \emph{permissible dual} fibration \cite[Definition~$2.3$]{G01I} $\breve{f} \colon \breve{X}(B) \to B$. $X(B)$ is simply connected follows by Proposition~\ref{pro:simply_conn}. Since $B$ has simple singularities $f$ has a permissible dual $\breve{X}(B)$, locally exchanging positive and negative compactifications. Moreover, adapting the argument in Proposition~\ref{pro:simply_conn}, $\breve{X}(B)$ is simply connected.
\end{proof}

\subsection{The topology of $X(P,D)$}
\label{sec:easy_topology}

We now turn to an analysis of the topology of the fibration $f \colon X(B) \to B$ in the special case that $B$ is equal to $B(P,D)$ and $X$ is equal to $X(P,D)$ for some s.d.\ polytope $P$ and standard decomposition $D$.

\begin{pro}
	\label{pro:simply_conn}
	The topological space $X$ is simply connected.
\end{pro}
\begin{proof}
	Following \cite[Theorem~$2.12$]{G01}, since $B = B(P,D)$ is simply connected we only need to check that $H^0(B,R^1f_\star\ZZ_n) = 0$ for all $n$. Since the compactified torus fibration $f \colon X(P,D) \to B$ is $\ZZ_n$-simple, $H^0(B,R^1f_\star\ZZ_n) = H^0(B,\iota_\star R^1f_0\ZZ_n)$, where $f_0$ denotes the restriction of $f$ to the preimage of $B_0$. Fixing a basepoint $b \in B_0$, we may identify this group with the subspace of $H^1(f^{-1}(b),\ZZ_n) \cong \ZZ_n^3$ fixed under the monodromy operators induced by loops passing around any segment of $\Delta$.
	
	Note that there is a canonical isomorphism $H^1(f^{-1}(b),\ZZ) \cong \Lambda_b$. Hence, if $l \subset \Delta$ is a segment contained in a face $\sigma \in \Faces(P^\circ,2)$, the monodromy operator corresponding to a loop passing singly around $l$ fixes the integral points in the tangent space to $\sigma$ (transported to $b$). Let $b$ be a vertex of $P^\circ$, and let $\{\sigma_i : i \in \{1,2,3\}\}$ be three elements of $\Faces(P^\circ,2)$ containing $b$ which do not all share a common edge. The tangent spaces of these faces at $b$ are the invariant planes for loops around segments of $\Delta$. The intersection of these monodromy invariant subspaces is trivial and, noting that $H^1(f^{-1}(b),\ZZ_n) = H^1(f^{-1}(b),\ZZ) \otimes_\ZZ \ZZ_n$, the result follows.
\end{proof}

\begin{pro}
	\label{pro:euler_number}
	The topological Euler number of $X$ is given by the formula
	\[
	\chi(X) = \sum_{\tau \in \Edges(P^\circ)}{\big(\#\{m \in D(\tau^\star) : \dim(m) = 2\}\big)} - \sum_{\sigma \in \Faces(P^\circ,2)}{\big(\ell(\sigma^\star)^2\Vol(\sigma)\big)},
	\]
	where $\Vol(A)$ is the volume of the polygon $A$, normalised so that the volume of a standard simplex is equal to $1$.
\end{pro}
\begin{proof}
	By Lemma~\ref{lem:euler_number} $\chi(X)$ is equal to the difference between the number of positive and negative vertices in $B$. Positive vertices occur precisely where $3$ segments of $\Delta$ meet at a point contained in an edge of $P^\circ$. Such points are in bijection with triangles appearing in $D$. Negative vertices are trivalent points of $\Delta$ contained in the relative interior of a $2$-dimensional face of $P^\circ$. Fixing a face $\sigma \in \Faces(P^\circ,2)$, negative vertices in $\sigma$ are in bijection with the triangles in a maximal triangulation of $\ell(\sigma^\star)\sigma$. This is precisely $\Vol(\ell(\sigma^\star)\sigma) = \ell(\sigma^\star)^2\Vol(\sigma)$; from which the result follows.
\end{proof}

\subsection{Computing $b_2(X(P,D))$}
\label{sec:compute_b2}

We now turn to the computation of the second Betti number of $X = X(P,D)$; closely related to the value $\gamma(P,D)$, see Definition~\ref{dfn:gammaPD}.

\begin{thm}
	\label{thm:rank}
	The second Betti number of $X$ is equal to $\gamma(P,D) - 3$.
\end{thm}

\begin{rem}
	Assuming that $X$ is homotopy equivalent to a Calabi--Yau threefold $X$, the Picard rank of $X$ is equal to $b_2(X)$. Indeed, it follows immediately from the exponential sequence for $X$ that the first Chern class $\Pic(X) \to H^2(X,\ZZ)$ is an isomorphism.
\end{rem}

The proof of Theorem~\ref{thm:rank} makes use of a contraction map $\bar{\xi} \colon X \to X_0$ (recalling that $X_0 = X_0(P,D)$ is the toric log Calabi--Yau space obtained from $(P,D)$). We define the map $\bar{\xi}$ locally; and show that the map $f$ factors as $\pi \circ \bar{\xi}$, where $\pi\colon X_0 \to B$ restricts to the moment map on each toric stratum of $X_0$.

\begin{cons}
	\label{cons:contraction}
	Given a point $b \in B_0$ such that the minimal stratum $\sigma$ of $\P$ containing $b$ has dimension $d$, the fibre $f^{-1}(b)$ is equal to $T^\star_b B/\breve{\Lambda}$. There is a canonical inclusion $T_b\sigma \rightarrow T_bB$, inducing a projection $T_b^\star B \rightarrow T^\star_b\sigma$. This projection descends to $f^{-1}(b)$, and maps $f^{-1}(b)$ to a possibly lower dimensional torus; the quotient of $T^\star_b\sigma$ by the restriction of $\breve{\Lambda}$. This determines a a map
	\[
	\bar{\xi}_0 \colon f^{-1}(B_0) \rightarrow X_0(B)
	\]
	which we now compactify over $\Delta$. Indeed, given a point $b' \in \Delta$, every vanishing cycle of the fibre $f^{-1}(b')$ is contained in the kernel of the projection  $T_b^\star B \rightarrow T^\star_b\sigma$, where $b$ is a general point of $B_0$ close to $b'$. Thus we can extend $\bar{\xi}_0$ over $\Delta$: this can be realized explicitly by defining torus actions on the fibres of $f$, following~\cite{G01}.
\end{cons}

We describe the possible fibres of $\bar{\xi}$ over points in $X_0$.
\begin{enumerate}
	\item If $x \in X_0$ and $x \notin \bar{\xi}(\Crit(f))$, then $\bar{\xi}^{-1}(x)$ is a torus of dimension $3-d$, where $d$ is the dimension of the minimal stratum of $\P$ containing $x$.
	\item If $x \in \bar{\xi}(\Crit(f))$ and $x \notin X_\tau$ for any edge $\tau$ of $\P$, then $\bar{\xi}^{-1}(x)$ is a point.
	\item If $x \in \bar{\xi}(\Crit(f))$ and $x \in X_\tau$ for some edge $\tau$ of $\P$, then $\bar{\xi}^{-1}(x)$ is a point if $\pi(x)$ is a positive vertex, while $\bar{\xi}^{-1}(x) \cong S^1$ if $\pi(x)$ is not a trivalent point of $\Delta$.
\end{enumerate}

We compute the cohomology of $X$ from the Leray spectral sequence. However, rather than directly applying the Leray spectral sequence associated to $\bar{\xi}$ we first compose it with a further contraction map.

\begin{cons}
	Let $\P'$ be a refinement of $\P$ to a simplicial complex such that the vertex set of $\P'$ is identical to the vertex set of $\P$. Such a refinement is uniquely determined by a choice of diagonal in each rectangular face of $\P$. Let $X'_0$ denote the corresponding (reducible) union of toric varieties, and let $\eta \colon X_0 \to X'_0$ be the corresponding contraction map. Outside of the strata of $X'_0$ which correspond to new faces of $\P'$, the map $\eta$ is a homeomorphism. Over a point  $x$ in the relative interior of a new stratum, $\eta^{-1}(x) \cong S^1$.	We let $\xi$ denote the composition $\eta \circ \bar{\xi}$. Let $C \cong S^2$ denote the restriction of $\bar{\xi}(\Crit(f))$ to $X_\sigma$, where $\sigma$ is a rectangular face of $\P$ subdivided in $\P'$ into faces $\sigma_1$ and $\sigma_2$. The image $\eta(C)$ is also a sphere, which intersects each of $X_{\sigma_1}$ and $X_{\sigma_2}$ in a disc. 
\end{cons}

Let $\tau$ be an edge of $\P'$ which is not an edge of $\P$, and let $\sigma_1$ and $\sigma_2$ be the new $2$-dimensional faces of $\P'$ containing $\tau$. If $x \notin \xi(\Crit(f))$ then -- as in our analysis of the map $\xi$ -- the fibres of $\xi^{-1}(x)$ are tori of dimension $(3-d)$, where $d$ is the dimension of the minimal stratum of $\P$ containing $x$. Letting 
\[
D_1 := \xi(\Crit(f)) \cap X_{\sigma_1},
\]
we have that $\xi^{-1}(x)$ is a pinched torus if $x$ i contained in $X_\tau \cap D_1$, while $\xi^{-1}(x)$ is a point if $x$ is contained in $(X_{\sigma_1}\setminus X_\tau) \cap D_1$.

We analyse the Leray spectral sequence $H^i(X'_0,R^j\xi_\star\QQ) \Rightarrow H^{i+j}(X,\QQ)$ for the map $\xi$. We first describe the groups in low degree on the $E_2$ page. To do so, we introduce maps $i_k$ for $k \in \{1,2,3\}$, generalising maps appearing in the proof of \cite[Theorem~$4.1$]{G01}. Let $F_k$ denote the disjoint union of toric codimension $k$ strata of $X'_0$. We let $i_k$, for $k \in \{0,\ldots,3\}$, denote the collection of canonical inclusions $F_k$ into $X'_0$. Note that, while $i_k$ is injective on any connected component of $F_k$, $i_k$ is not generally an injective function.

\begin{pro}
	\label{pro:E2_page}
	The $E_2$ page of the Leray spectral consists of the following groups in low degree. Consequently $b_2(X)$ is equal to $1+\dim(\ker d)$.
\[
	\xymatrix@R-2pc{
	\QQ & & & \\
	0 & \star & & \\
	0 & H^1(R^1\xi_\star\QQ)\ar^(0.4)d[drr] & \star & \\
	\QQ & 0 & \QQ & \QQ \\
	}
\]
\end{pro}
\begin{proof}
	We follow the the proof of Theorem~$4.1$ of \cite{G01}; noting that a similar calculation was made in \cite[Proposition~$7.9$]{P:Fibrations}. First observe that $R^3\xi_\star\QQ = \bigoplus_{v \in \V{P^\circ}}\QQ_v$, the direct sum of skyscraper sheaves over the $0$-dimensional strata of $P^\circ$. Thus,
	\[
	H^0(X'_0,R^3\xi_\star\QQ) \cong \QQ.
	\]
	Second, we consider the map
	\[
	R^2\xi_\star\QQ \rightarrow {i_2}_\star{i_2}^\star R^2\xi_\star\QQ.
	\]
	Following the argument used in \cite{G01}, this map is monomorphic and there is an inclusion
	\[
	H^0(R^2\xi_\star\QQ) \hookrightarrow H^0({i_2}_\star i_2^\star R^2\xi_\star\QQ).
	\]
	The sheaf ${i_2}_\star i_2^\star R^2\xi_\star\QQ$ is nothing but the direct sum of its restrictions to the one dimensional toric strata of $X'_0$ (corresponding to edges $\tau$ of $\P'$). Each such stratum $X_\tau$ is isomorphic to $\PP^1$. The restriction of $ {i_2}_\star{i_2}^\star R^2\xi_\star\QQ$ of $X_\tau$ is isomorphic to the constant sheaf $\QQ$ away from either a finite set of points, or a circle of points which have trivial stalks. The first case applies to edges $\tau \in \P$, the second to edges introduced in $\P'$. Fix a tuple of sections 
	\[
	s := (s_\tau : \tau \in \Edges(\P')) \in H^0({i_2}_\star i_2^\star R^2\xi_\star\QQ).
	\]
	Each component $s_\tau$ of $s$ can be identified with an element of $H^2(\xi^{-1}(x),\QQ)$ for a point $x \in X_\tau$. This vector space is canonically isomorphic to $T^\star_v\tau$ for any $v \in \V{\tau}$. Note that $s_\tau = 0$ for any $\tau$ such that $\tau \cap \Delta \neq \varnothing$. Now assume that $s$ defines a section of $H^0(R^2\xi_\star\QQ)$. Thus, for any vertex $v$ of $\P$, the elements $s_\tau$ for edges $\tau$ incident to $v$ are all obtained by projections 
	\[
	T^\star_vB \to T^\star_v\tau.
	\]
	Let $R$ denote the restriction of $\xi(\Crit(f))$ to $X_\tau$. We have that $H^0(X_\tau,\QQ_{X_\tau \setminus R}) \cong H^0_c(X_\tau \setminus R,\QQ) \cong \{0\}$. Hence, letting $v$ be a vertex of $P^\circ$, all such projections vanish (since $s_\tau = 0$ for any $\tau$ such that $\tau \cap \Delta \neq \varnothing$); hence the section $s$ of $R^2\xi_\star\QQ$ vanishes at every vertex of $P^\circ$. This implies that $s$ vanishes at every vertex of $\P$, and hence $s=0$. A similar argument, applied to the map
	\[
	R^1\xi_\star\QQ \rightarrow {i_1}_\star{i_1}^\star R^1\xi_\star\QQ.
	\]
	shows that $H^0(X'_0, {i_1}_\star i_1^\star R^1\xi_\star\QQ)$ vanishes. Indeed, letting $C$ denote the restriction to $X_\sigma$ for a face $\sigma \in \P'$ we have that $X_\sigma$ is a weighted projective plane and $C$ is either a sphere, or a disc. In either case, $H^0_c(X_\sigma\setminus C,\QQ) = \{0\}$.

	We now consider the cohomology groups $H^\bullet(X'_0,\xi_\star\QQ)$. Note that, since every fibre of $\xi$ is connected, we have that
	\[
	\xi_\star\QQ \cong \QQ.
	\]
	That is, these cohomology groups are nothing other than the ordinary rational cohomology groups of $X'_0$. Following the proof of \cite[Theorem~$4.1$]{G01}, we use the spectral sequence associated to the decomposition of $X'_0$ into its maximal toric strata. Recall that the underlying complex of the decomposition of $B$ is homeomorphic to $S^3$, and that $H^0(Y,\QQ) \cong H^2(Y,\QQ) \cong \QQ$ for each $3$-dimensional toric stratum $Y$ of $X'_0$.

The bottom row of the $E_1$ page of the spectral sequence associated to the decomposition of $X'_0$ consists of the exact sequence associated to the \u{C}ech complex of the intersection graph of $X'_0$, see \cite[p.$47$]{G01}, which has the form
\begin{equation}
\label{eq:int_complex}
\xymatrix{
\QQ^{n_3} \ar^{d_3}[r] & \QQ^{n_2} \ar^{d_2}[r] & \QQ^{n_1} \ar^{d_1}[r] & \QQ^{n_0},
}
\end{equation}
where $n_i$ records the numbers of $i$-dimensional cells of $\P$ for each $i \in \{0,1,2,3\}$. The odd numbers rows of the $E_1$ page vanish, while the row $E_1^{\bullet,2}$ is the truncation of \eqref{eq:int_complex} to its first three terms. Indeed, for each face $\sigma$ of $\P'$, the toric variety $X_\sigma$ is either a weighted projective space, and $H^2(X_\sigma,\QQ) \cong \QQ$.

Since $X_0$ is projective $X'_0$ is projective and, similarly to \cite[p.$47$]{G01}, $H^2(X'_0,\QQ)$ cannot vanish. It follows that 
\[
H_2(X'_0,\QQ) \cong H_3(X'_0,\QQ) \cong \QQ.
\]
\end{proof}

\begin{lem}
	\label{lem:vanishing_diff}
	Writing $X'_0$ for the topological space obtained by degenerating $X_0(P,D)$ as above, the map
	\[
	d \colon H^1(X'_0,R^1\xi_\star\QQ) \to H^3(X'_0,\xi_\star\QQ),
	\]
	which appears in the statement of Proposition~\ref{pro:E2_page}, vanishes. Hence, by Proposition~\ref{pro:E2_page}, $b_2(X) = h^1(B,R^1\xi_\star\QQ)+1$.
\end{lem}
\begin{proof}
	This follows from the argument used in \cite[Lemma~$2.4$]{G01I} to prove that the Leray spectral sequence associated to $f$ degenerates at the $E_2$ page. In particular, we recall that $d$ fits into an exact sequence
	\[
	 H^1(X'_0,R^1\xi_\star\QQ) \to H^3(X'_0,\xi_\star\QQ) \to E^{0,3}_\infty \to H^3(X,\QQ).
	\]
	Moreover the composition of the latter two maps is the pullback $\xi^\star \colon H^3(X'_0,\xi_\star\QQ) \to H^3(X,\QQ)$. The topological torus fibration $f$ factors through $\xi$; indeed, $f = \bar{\pi}\circ \xi$, where $\bar{\pi}$ is the collection of moment maps of maximal toric strata of $X'_0$. The map $f$ admits a section which factors through $\xi$. Hence the image of $\xi^\star$ has positive dimension, and the image of $d$ is trivial. 
\end{proof}

To compute $h^1(X'_0,R^1\xi_\star\QQ)$ we consider the inclusion of the disjoint union of two dimensional toric strata ${\iota_1}_\star\iota_1^\star R^1\xi_\star\QQ$. Let $\cF$ denote the cokernel of the monomorphism
\[
R^1\xi_\star\QQ \to {\iota_1}_\star\iota_1^\star R^1\xi_\star\QQ.
\]
In particular, $\cF$ fits into a short exact sequence; analysing the corresponding long exact sequence we obtain a relation between $H^0(X'_0,\cF)$ and $H^1(X'_0,R^1\xi_\star \QQ)$. To state this, we define $t(P,D)$ to be the number of $2$-dimensional faces of $P^\circ$ disjoint from $\Delta$.

\begin{lem}
	\label{lem:describing_H1}
	The space $H^0(X'_0, {\iota_1}_\star\iota_1^\star R^1\xi_\star\QQ)$ has dimension $t(P,D)$. Moreover, there is an equality
	\[
	 h^0(X'_0,\cF) = h^1(X'_0,R^1\xi_\star \QQ) + t(P,D).
	\]
\end{lem}
\begin{proof}
	Recall the equality of sheaves
	\[
	{\iota_1}_\star\iota_1^\star R^1\xi_\star\QQ = \bigoplus_{\sigma \in \Faces(\P',2)} \QQ|_{X_\sigma \setminus C},
	\]
	where $C := \xi(\Crit(f)) \cap X_\sigma$. First assume that $\sigma \cap \Delta \neq \varnothing$; this case separates into two sub-cases.
	\begin{enumerate}
		\item $X_\sigma \cong \PP^2$ and $C$ is a sphere in the class of a complex projective line. In this case $H^i(X_\sigma, \QQ_{X_\sigma \setminus C}) = H^i_c(X_\sigma \setminus C, \QQ)$. By Poincar\'e duality this is isomorphic to the group $H_{4-i}(X_\sigma \setminus C, \QQ)$, which vanishes for each $i \in \{0,1\}$.
		\item  $X_\sigma$ is a weighted projective plane, and $C$ is a disk. In this case the complement retracts onto a projective line, and again $H^i$ vanishes for $i \in \{0,1\}$.
	\end{enumerate}
	
	Next assume that $\sigma \cap \Delta = \varnothing$. In this case $X_\sigma$ is a weighted projective plane, and $H^0(X_\sigma,\QQ) \cong \QQ$ and $H^1(X_\sigma,\QQ) \cong \{0\}$. To prove the final equality in the statement of Lemma~\ref{lem:describing_H1}, we study following exact sequence
	\begin{align*}
	H^0(X'_0, R^1\xi_\star\QQ) &\to H^0(X'_0, {\iota_1}_\star\iota_1^\star R^1\xi_\star\QQ) \to H^0(X'_0, \cF) \to \\
	\to H^1(X'_0, R^1\xi_\star\QQ) &\to H^1(X'_0, {\iota_1}_\star\iota_1^\star R^1\xi_\star\QQ) .
	\end{align*}
	By Proposition~\ref{pro:E2_page}, $H^0(X'_0, R^1\xi_\star\QQ)$ vanishes. Moreover, the above analysis of $\QQ_{X_\sigma \setminus C}$ for each face $\sigma \in \Faces(P^\circ,2)$ shows that $ H^1(X'_0, {\iota_1}_\star\iota_1^\star R^1\xi_\star\QQ) = \{0\}$; from which the result follows.
\end{proof}

To compute the dimension of $H^0(X'_0,\cF)$, we make use of an inverse limit of a system of vector spaces associated to $(P,D)$, defined as follows. Note that, for any polytope $\rho$ in $N_\RR$, we can define a subspace $T\rho \subset N_\RR$ by translating $\rho$ so that $\rho$ contains the origin and taking the linear span. Equivalently, $T\rho$ is the minimal vector subspace of $N_\RR$ containing some translate of $\rho$.

\begin{enumerate}
	\item Observe that the set 
	\[
	\Edges(P)~{\mbox{\larger[-2]$\coprod$}}~~\coprod_{\rho \in \Faces(P,2)}{D(\rho)} 
	\]
	is partially ordered by inclusion (using the canonical matching of edges of each $m \in D(\rho)$ and edges of $\rho$).
	\item To each edge $E$ of $P$ we associate the vector space $(TE \cap N_\QQ)^\star$.
	\item To each $m \in D(\rho)$ we associate the vector space $(Tm \cap N_\QQ)^\star$.
	\item We define an inverse system by associating the map dual to the inclusion $TE \to Tm$ if $E \leq m$ in this partially ordered set.
\end{enumerate}

\begin{dfn}
	\label{dfn:gammaPD}
	The inverse limit of the above system is a $\QQ$-vector space $\Gamma(P,D)$. We let $\gamma(P,D)$ denote the dimension of $\Gamma(P,D)$.
\end{dfn}

\begin{rem}
	Note that an element of $\Gamma(P,D)$ is determined by its projection to vector spaces associated to the edges of $P$. Equivalently, we may consider these to be vectors associated to elements of $\Faces(P^\circ,2)$, and the condition of lying in the inverse limit imposes conditions on these vectors. In other words, we may identify $\Gamma(P,D)$ with a subspace of $\bigoplus_{\sigma \in \Faces(P^\circ,2)}{T_vB/T_v\sigma}$ where $v$ is a vertex of each face $\sigma$. Note that we may drop the dependence on $v$ since 
	\[
	T_vB/T_v\sigma \cong M_\RR/\langle \sigma \rangle \cong (T\sigma^\star \cap N_\QQ)^\star;
	\]
	where $\langle \sigma \rangle$ denotes the linear span of (the cone over) $\sigma$.
\end{rem}

To compute $h^0(X'_0,\cF)$ we make use of an additional sheaf $\cG$ on $X'_0$. This is defined to be such that $\cG_x = M_\QQ/\langle \sigma \rangle$, where $\sigma$ is the minimal face of $\P'$ such that $x \in X_\sigma$. We show that there is a natural map $\cF \to \cG$, and describe the induced morphism between stalks of $\cF$ and $\cG$. This analysis is similar to that pursued in \cite[\S$7$]{P:Fibrations}.

We first describe the stalks of the sheaf $\cF$. Fix an edge $\tau$ of $\P'$ and let $\sigma_1,\ldots, \sigma_k \in \Faces(\P',2)$ denote the faces of $\P'$ meeting $\tau$. Given a point $q$ on the projective line corresponding to $\tau$ not contained in the singular locus, $\cF_q$ is the cokernel of the map
\[
H^1(\xi^{-1}(q),\QQ) \cong \QQ^2 \rightarrow \bigoplus_{1 \leq j \leq k} H^1(\xi^{-1}(q_j),\QQ) \cong \QQ^k,
\]
where $q_j \in X_{\sigma_j} \setminus X_{\tau}$ are points close to $q$ for each $j \in \{1,\ldots,k\}$. Suppose now that $q$ is the image under $\xi$ of the (unique) singular point of the fibre lying over a positive node of $B$, and let $j_1$, $j_2$, and $j_3$ in $\{1,\ldots,k\}$ be the indices of the faces whose interiors intersect the singular locus in any neighbourhood of the image of $q$ in $B$. Setting $I_q := \{1,\ldots, k\} \setminus \{j_1,j_2,j_3\}$, $\cF_q$ is the cokernel of the map
\[
H^1(\xi^{-1}(q),\QQ) \cong \{0\} \rightarrow \bigoplus_{j \in I_q} H^1(\xi^{-1}(q_j),\QQ) \cong \QQ^{k-3}
\]
where the second sum is over faces meeting $\tau$ which do not meet singular locus near $q$. Suppose finally that $q$ lies over a general point of the singular locus of $B$ contained in $\tau \in \Edges(\P')$ and $q$ is contained in the image (under $\xi$) of the circle of singularities of this fibre. Then $\cF_q$ is the cokernel of the specialization map
\[
H^1(\xi^{-1}(q),\QQ) \cong \QQ^1 \rightarrow \bigoplus_{j \in I_q} H^1(\xi^{-1}(q_j),\QQ) \cong \QQ^{k-2}
\]
where, again, $I_q$ indexes faces $\sigma_j$ for $j \in \{1,\ldots,k\}$ such that the interior of $\sigma_j$ does not intersect the singular locus in some neighbourhood of the image of $q$ in $B$. This map factors as 

\begin{equation}
\label{eq:factorise}
\QQ^1 \to H^1(\xi^{-1}(q'),\QQ) \to \bigoplus_{j \in I_q} H^1(\xi^{-1}(q_j),\QQ) \cong \QQ^{k-2}
\end{equation}

where $q' \in X_\tau$ is a nearby point not contained in $\xi(\Crit(f))$. The segment of $\Delta$ containing $\tau$ determines the class of a vanishing cycle $\alpha \in H_1(f^{-1}(b),\QQ)$, where $b \in B_0$ is the image of $q'$ in $\tau$. The space $H_1(\xi^{-1}(q'),\RR)$ is canonically isomorphic to the annihilator of $\tau$ in $T^\star_bB$, which contains $\alpha$. The map $\QQ \to H^1(\xi^{-1}(q'),\QQ)$ above is (the restriction to rational points of) the dual map to the quotient map $\Ann(T_b\tau) \to \Ann(T_b\tau)/\langle \alpha \rangle$.

\begin{rem}
 Note that this analysis holds when $q \in X_\tau$, and $\tau \in \Edges(\P')$ is not an edge of $\P$, although in this case there are a circle of points $q \in X_\tau \cap \xi(\Crit(f))$. The difference between this case and the case $\tau \in \Edges{\P}$ can be interpreted as the fact that, while $\alpha \in \Ann(T_b\tau)$ in either case, this subspace is only the monodromy invariant plane in $H_1(f^{-1}(b),\QQ)$ for loops around the given segment of $\Delta$ if $\tau \in \Edges(\P)$.
\end{rem}

We now consider the maps of stalks $\cF_q \to \cG_q$. These are defined by the diagram.
\begin{equation}
\label{eq:cFprime}
\xymatrix{
	H^1(\xi^{-1}(q),\QQ) \ar[r] \ar[d] & \bigoplus\limits_{j \in I_q} H^1(\xi^{-1}(q_j),\QQ) \ar[r] \ar[d] & \cF_{q} \ar[d] \\
	H^1(\xi^{-1}(q'),\QQ) \ar[r] & \bigoplus\limits_{1 \leq j \leq k} H^1(\xi^{-1}(q_j),\QQ) \ar[r] & \cG_q
}
\end{equation}

In the case of a positive node (a two-dimensional factor in $D(\tau^\star)$), the diagram~\eqref{eq:cFprime} becomes:
\begin{equation}
\label{eq:positive}
\xymatrix{
	\{0\} \ar[r] \ar[d] & \QQ^{k-3} \ar[r] \ar[d] & \cF_q \ar[d] \\
	\QQ^2 \ar[r] & \QQ^k \ar[r] & \cG_{q}
}
\end{equation}
That is, the condition imposed on $\QQ^k$ by a $2$-dimensional element $m \in D(\tau^\star)$ is that values on the factors corresponding $\sigma_j$, where $\sigma_j^\star \in \Edges(\tau^\star,m)$ sum to an element of $H^1(\xi^{-1}(q'),\QQ)$, where $q' \in X_\tau$ is the point appearing in \eqref{eq:factorise}. In the second case, $\dim m = 1$, and the diagram~\eqref{eq:cFprime} becomes
\begin{equation}
\label{eq:generic}
\xymatrix{
	\QQ \ar[r] \ar[d] & \QQ^{k-2} \ar[r] \ar[d] & \cF_{q} \ar[d] \\
	\QQ^2 \ar[r] & \QQ^k \ar[r] & \cG_{q}.
}
\end{equation}
In other words, the two components corresponding to faces $\sigma$ which intersect $\Delta$ near $q$ must sum to zero.

\begin{lem}
	\label{lem:injective}
	The canonical map $\cF \to \cG$ is a monomorphism.
\end{lem}
\begin{proof}
	This follows directly from our analysis of the commutative diagram \eqref{eq:cFprime}. If $q$ is not in $\xi(\Crit(f))$ then $\cF_q \to \cG_q$ is an isomorphism, and hence injective. If $q \in \xi(\Crit(f))$ lies over a positive node of $\Delta$, then the morphism $\cF_q \to \cG_q$ is described in \eqref{eq:positive}. An element in the kernel of this map lifts to an element of $\QQ^{k-3}$ whose image in $\QQ^k$ lies in the image of $\QQ^2$. However, the images of $\QQ^{k-3}$ and $\QQ^2$ intersect trivially. Similarly, if $q \in \xi(\Crit(f))$ lies over a general point in $\Delta$, then an element in the kernel of $\cF_q \to \cG_q$ lies in the intersection of both $\QQ^{k-2}$ and $\QQ^2$. However these intersect in a one dimensional space, the image of $\QQ \to \QQ^{k-2}$.
\end{proof}

\begin{lem}
	\label{lem:cohomology_vanishes}
	Given the $\cG$ on $X'_0$ as above, we have that
	\begin{align*}
	H^0(X'_0,\cG) \cong M_\QQ&, \text{ and} \\
	H^1(X'_0,\cG) = \{0\}&.
	\end{align*}
\end{lem}
\begin{proof}
	There is a \u{C}ech-to-derived spectral sequence associated to the decomposition of $X'_0$ into toric varieties given by the maximal cells of $\P'$. The bottom row of this complex has the form
	\begin{equation}
	\label{eq:toric_complex}
	0 \to \bigoplus_{v \in \V{\P'}} M_\QQ/\langle v \rangle \to \bigoplus_{\tau \in \Edges(\P')} M_\QQ/\langle \tau \rangle \to \bigoplus_{\sigma \in \Faces(\P',2)} M_\QQ/\langle \sigma \rangle.
	\end{equation}
	Moreover, since $H^1(X_\sigma,\QQ) = \{0\}$ for all faces $\sigma$ of $\P'$, we have that $H^1(X'_0,\cG)$ is isomorphic to the first cohomology group of \eqref{eq:toric_complex}.
	
	Taking cones over the faces of $\P'$ defines a simplicial fan $\Sigma$, and a corresponding toric variety $X_\Sigma$. There is a spectral sequence \cite[\S$12.3$]{CLS11},
	\[
	E_1^{p,q} = H_c^{p+q}(X_p,X_{p-1},\QQ) \Rightarrow H_c^{p+q}(X_\Sigma,\QQ),
	\]
	where $X_p$ denotes the union of the $p$-dimensional toric strata of $X_{\Sigma}$. Moreover, the row $(E^{p,\bullet}_1,d)$ of this sequence is dual to \eqref{eq:toric_complex} (after replacing the left-most zero of \eqref{eq:toric_complex} with $M_\QQ$). Applying \cite[Proposition~$12.3.10$]{CLS11}, this spectral sequence degenerates at the $E_2$ page and hence $H^1(X'_0,\cG)$ is dual to a graded piece of $H^3(X_\Sigma,\QQ)$. However the latter group vanishes by \cite[Theorem~$12.3.11$(a)]{CLS11}. Similarly, $E^{3,1}_2$ is dual to the quotient of the first cohomology group of \eqref{eq:toric_complex} by the image of the canonical map
	\[
	M_\QQ \to \bigoplus_{v \in \V{\P'}} M_\QQ/\langle v \rangle
	\]
	
	This quotient vanishes by \cite[Theorem~$12.3.11$(b)]{CLS11}. 
\end{proof}

\begin{lem}
	\label{lem:global_sections}
	Given $X'_0$ and $\cF$ as above, we have an equality
	\[
	 h^0(X'_0,\cF) = \gamma(P,D) + t(P,D) - 4.
	\]
\end{lem}
\begin{proof}
	
	We first compare the sheaves $\cF$ and $\cG$. Let $\bar{\cF}$ denote the cokernel
	\[
	\bar{\cF} := \coker\big(\cG \to  {\iota_1}_\star\iota_1^\star \cG \big),
	\]
	and form the commutative diagram
	\[
	\xymatrix{
	 0 \ar[r] \ar[d] & H^0(X'_0, {\iota_1}_\star\iota_1^\star R^1\xi_\star\QQ) \ar[r] \ar[d] & H^0(X'_0, \cF) \ar[r]\ar^{\theta}[d]  &  H^1(X'_0, R^1\xi_\star\QQ) \ar[r]\ar[d]  & 0 \\
	H^0(X'_0,\cG) \ar[r] & H^0(X'_0, {\iota_1}_\star\iota_1^\star \cG) \ar[r] & H^0(X'_0, \bar{\cF}) \ar[r] &  0  & 
	}
	\]
	noting that $H^1(X'_0,\cG)$, $H^0(X'_0,R^1\xi_\star\QQ)$, and $H^1(X'_0, {\iota_1}_\star\iota_1^\star R^1\xi_\star\QQ)$ all vanish by Lemma~\ref{lem:cohomology_vanishes}, Proposition~\ref{pro:E2_page}, and Lemma~\ref{lem:describing_H1} respectively. By Lemma~\ref{lem:injective} and left exactness of $H^0$, $\theta$ is injective. Thus, applying Lemma~\ref{lem:cohomology_vanishes}, $H^0(X'_0,\cF)$ may be viewed as a subspace of $V/M_\QQ$, where
	\[
	V := \bigoplus_{\sigma \in \Faces(P^\circ,2)}{T_vB/T_v\sigma}
	\]
	is canonically isomorphic to $H^0(X'_0, {\iota_1}_\star\iota_1^\star \cG)$. Note that the inclusion $M_\QQ \to V$ is defined by sending $u \in M_\QQ$ to the tuple of equivalence classes $[\pi_v(u)] \in T_vB/T_v\sigma$, where $\pi_v$ is the map defining the fan structure at $v \in \V{\P'} = \V{\P}$.
	
	It remains to show that the subspace of $V$ which maps onto the image of $\theta$ has dimension $t(P,D) + \gamma(P,D)$. Any element in $V$ determines a section of $\cF$ away from $\xi(\Crit(f))$. Thus, the conditions on $V$ which express $H^0(X'_0,\cF)$ as a subspace of $V/M_\QQ$ appear from extending this section over $\xi(\Crit(f))$. 
	
	The result now follows from the analysis of the stalks of $\cF$ and $\cG$ made above Lemma~\ref{lem:injective}. Indeed, any element in $H^0(X'_0,\bar{\cF})$ determines a section of $\cF$ away from the singular locus. Fix a face $\bar{\sigma} \in \Faces(P,2)$.  The condition imposed by \eqref{eq:generic} implies that if $x \in V$ determines an element of $H^0(X'_0,\cF)$, co-ordinates of $x$ corresponding to every face $\sigma \subset \bar{\sigma}$ such that $\bar{\sigma} \cap \Delta \neq \varnothing$ coincide; let $y_{\bar{\sigma}}$ denote this value. Moreover, the conditions defining the inverse limit $\Gamma(P,D)$, are precisely the conditions imposed by positive vertices, and non-trivalent points contained in the $1$-skeleton of $P^\circ$. In other words, the subspace of $V$ which maps onto the image of $\theta$ is generated by the direct sum of a subspace of dimensional $t(P,D)$ subspace (corresponding to faces which do not meet $\Delta$) with a subspace of dimension $\gamma(P,D)$ (corresponding to the subspace $\Gamma(P,D)$ of $\{y_\rho : y_\rho \in \QQ, \rho \in \Faces(P^\circ,2)\}$).
\end{proof}

Combining the above lemmas, we complete our combinatorial description of the second Betti number of $X$. 

\begin{proof}[Proof of Theorem~\ref{thm:rank}]
	From the analysis of the Leray spectral sequence made in Proposition~\ref{pro:E2_page}, 
	\[
	h^2(X,\QQ) = h^1(X'_0, R^1\xi_\star\QQ) + 1.
	\]
	By Lemma~\ref{lem:describing_H1}, 
	\[
	h^1(X'_0,R^1\xi_\star \QQ) = h^0(X'_0,\cF) - t(P,D);
	\]
	while, by Lemma~\ref{lem:global_sections},  $h^0(X'_0,\cF) = \gamma(P,D) + t(P,D) - 4$. Thus, combining these results, $b_2(X) = \gamma(P,D) - 3$.
\end{proof}

\begin{eg}
	We compute an example of $b_2(X)$, using Theorem~\ref{thm:rank}. Let $P$ be the polytope in $N_\RR$ such that the toric variety defined by the spanning fan of $P$ is isomorphic to $\PP^2 \times \PP^2$. To fix notation, let $P_3$ the convex hull of $(1,0)$, $(0,1)$ and $(-1,-1)$; we identify $P$ with the convex hull of $P_3^1 := P_3 \times \{0\}$ and $P_3^2 := \{0\} \times P_3$. In this case $D$ is unique, as every $2$-dimensional face of $P$ isomorphic to the standard triangle, and we set $X := X(P,D)$. Note that we expect that $b_2(X) = 2$, and hence that $\gamma(P,D) = 5$.
	
	Fixing orientations of the edges of $P$ determines an isomorphism $\Gamma(P,D) \cong \QQ^{\gamma(P,D)}$. We fix specific orientations for the edges of $P$ as follows. 
	\begin{enumerate}
		\item Fix orient the edges in the polygons $P_3^1$ and $P_3^2$ clockwise.
		\item Orient every other edge from its vertex in $P_3^1$ to its vertex in $P_3^2$
	\end{enumerate}
	Number the edges and vertices of $P_3^1$ and $P_3^2$ clockwise, such that the $i$th edge contains the $i$th and $(i-1)$th vertices. Let $x_i$ and $x^i$ denote the co-ordinates on $H'$ corresponding to edges of $P_3^1$ and $P_3^2$ respectively. Let $y_{i,j}$ denote the co-ordinate on $H'$ corresponding to the edge between the $i$th vertex of $P_3^1$ and the $j$th vertex of $P_3^2$. There are $15$ variables, and $18$ equations of the form
	\begin{align*}
	y_{i,j} &= y_{i-1,j} - x_i \\
	y_{i,j} &= y_{i,j-1} + x^j
	\end{align*}
	for $i,j \in \{1,2,3\}$. First eliminate the variables $x_i = y_{i-1,i} - y_{i,i}$ and $x^i = y_{i,i} - y_{i,i-1}$. The remaining $12$ equations are equivalent to a subset of $6$ equations. Solutions to these are given by $3\times 3$ matrices $(y_{i,j})$ for which the difference $y_{i_1,j} - y_{i_2,j}$ between elements in different rows and the same column is independent of $j$. Such matrices are uniquely determined by fixing all the elements in a single row and column. Thus we obtain a five dimensional space of solutions, and hence $\gamma(P,D) = 5$ and $b_2(X)=2$.
\end{eg}
\section{Joins, and Hadamard products}
\label{sec:examples}

In forthcoming work we classify simply decomposable $4$-dimensional reflexive polytopes by querying the Kreuzer--Skarke database of all $4$-dimensional reflexive polytopes. In this section we consider various standard decompositions of the product of two integral hexagons.

\begin{lem}
	There are $28$ four dimensional polytopes which may be obtained as the product of a pair of reflexive polygons.
\end{lem}
\begin{proof}
	The product of two lattice polygons is reflexive if and only if both polygons are themselves reflexive. It follows directly from the classification of ($16$) reflexive polygons that there are seven s.d.\ reflexive polygons; thus $28$ s.d.\ reflexive polytopes are the product of a pair of polygons.
\end{proof}

\begin{figure}
	\includegraphics[scale=0.3]{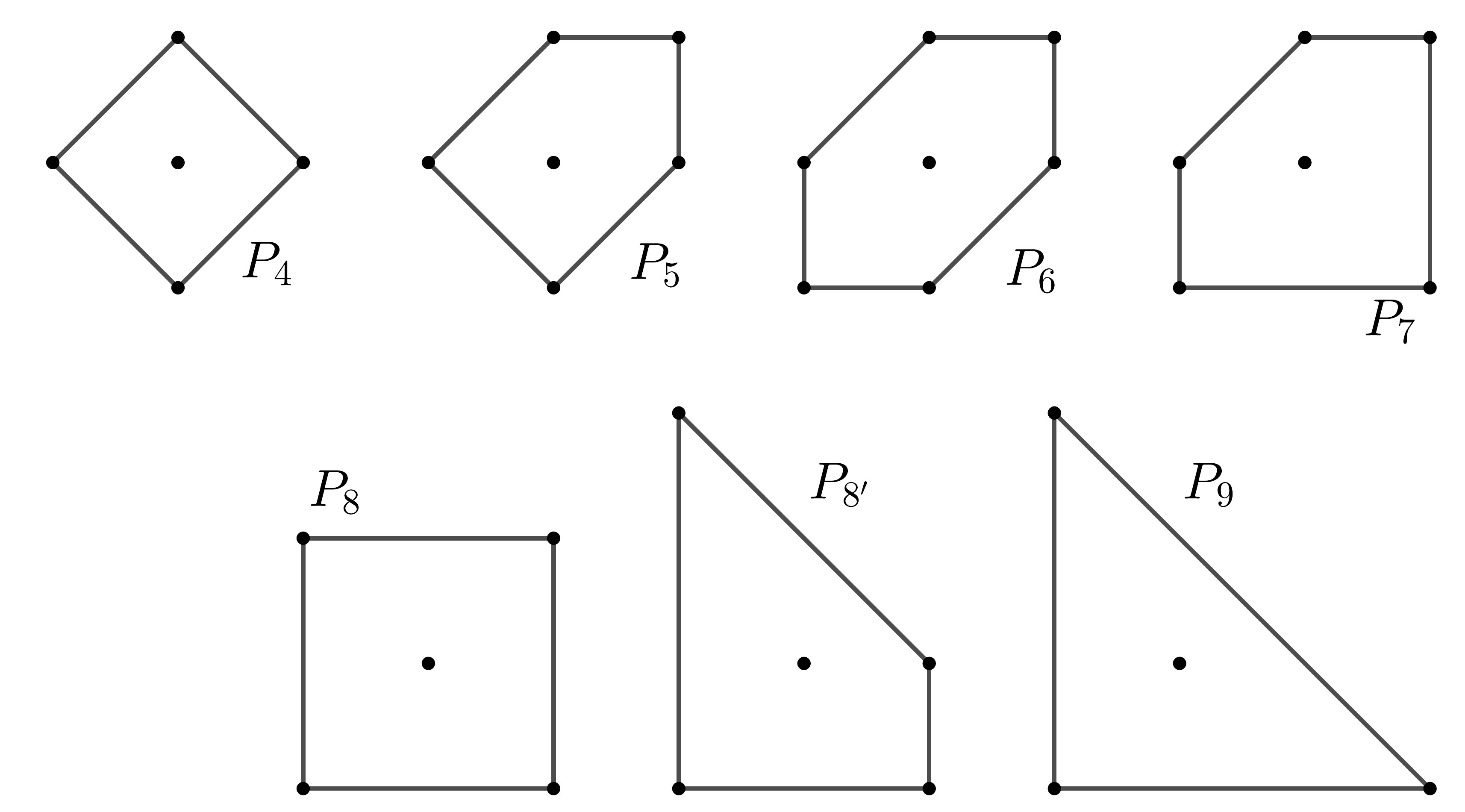}
	\caption{List of s.d.\ reflexive polygons.}
	\label{fig:reflexives}
\end{figure}

Let $P_k$ be the polygons displayed in Figure~\ref{fig:reflexives} and let $P_{k_1,k_2}$ denote the product $P_{k_1} \times P_{k_2}$. The toric varieties associated to these polytopes (via their \emph{spanning} or \emph{face} fan) are known to smooth in some cases.

\begin{eg}
	\label{eg:ci33}
	Setting $P := P_{9,9}$, we note that the toric variety $X_P$ is given by the equations
	\begin{align*}
	x_1x_2x_3 = x_0^3 && x_4x_5x_6 = x^3_0
	\end{align*}
	in $\PP^6$, the anti-canonical model of $X_P$. A general anti-canonical (hyperplane) section is then given by a $(3,3)$ complete intersection in $\PP^5$. We note that such a complete intersection $X$ has $b_2(X) = 1$ and $\chi(X) = -144$. We verify the compatibility of these statements with the results of \S\ref{sec:topology}. Note that there is a unique choice of Minkowski decompositions $D$ of the $2$-dimensional faces of $P$.
	
	The value $\chi(X)$ is given in Proposition~\ref{pro:euler_number}. There are a total of $18$ standard simplices in $D$. The number of negative vertices in $B(P,D)$ is equal to $9 \times \#\Edges(P) = 9 \times 18 = 162$. Hence observe that $-144 = 18 - 162$, as expected.
	
	To compute $b_2(X)$, we assign variables to the elements of $\Edges(P)$, and compute $\gamma(P,D)$ by imposing conditions associated to each $2$-dimensional face of $P$. Index the edges and vertices of $P_9$ clockwise with elements from $\{1,2,3\}$. Let $x_{i,j}$ be the variable corresponding to the product of the $i$th edge and $j$th vertex, and let $x^{i,j}$ be the variable corresponding to the product of the $j$th vertex and $i$th edge. The conditions imposed by the Minkowski factors of each element of $\Faces(P,2)$ depend on an orientation of the edges of $P$. We fix a cyclic orientation of the edges of $P_9$; this induces an orientation of every edge of $P = P_{9,9}$. The conditions imposed at the product of two edges are of the form $x_{i,j} = x_{i,j+1}$ and $x^{i,j} = x^{i,j+1}$, interpreted cyclically. Writing $x_i := x_{i,j}$ for all $j \in \{1,2,3\}$ and $x^i := x^{i,j}$ for all $j \in \{1,2,3\}$, the conditions imposed at triangular faces take the form
	\begin{align*}
	x_1+x_2+x_3 = 0 && x^1+x^2+x^3=0
	\end{align*}
	resulting in a $4$-dimensional space of solutions $\Gamma(P,D)$, and -- applying Theorem~\ref{thm:rank} -- we conclude that $b_2(X) = 1$.
\end{eg}

\begin{eg}
	\label{eg:other_egs}
	An analysis similar to that used in Example~\ref{eg:ci33} applies in a number of other cases;
	\begin{enumerate}
		\item The toric varieties determined by the spanning fans of $P_{8,9}$ and $P_{8',9}$ are $(2,2,3)$ complete intersections in $\PP^6$.
		\item  The toric varieties determined by the spanning fans of $P_{8,8}$, $P_{8,8'}$, and $P_{8',8'}$ are $(2,2,2,2)$ complete intersections in $\PP^7$.
		\item The toric variety determined by the spanning fan of $P_{7,9}$ smooths to a $(1,1,3)$ complete intersection in $\Gr(2,5)$.
		\item The toric variety determined by the spanning fan of $P_{7,8}$ smooths to a $(1,2,2)$ complete intersection in $\Gr(2,5)$.
	\end{enumerate}
	In all cases we may verify that Proposition~\ref{pro:euler_number} and Theorem~\ref{thm:rank} reproduce the expected invariants. Note that these examples are (somewhat trivial) instances of smoothing \emph{joins} of elliptic curves of degrees $3$, $4$, and $5$. 
\end{eg}

In fact several other possible examples have been very recently realised via constructions of smoothing components of joins of del Pezzo manifolds; see \cite{G15,I19,KS19}. For example, letting $P:= P_{7,7}$, and noting $D$ is uniquely determined by $P$, the space $X(P,D)$ has $(b_2,\chi) = (1,-100)$. Moreover, $\Vol(P^\circ) = 25$, implying that $H^3 = 25$ for generator $H$ of $H^2(X(P,D),\QQ)$. The ideal cutting out the toric variety $X_P$ in its anti-canonical embedding is the sum of ideals generated by $4\times 4$ Pfaffians. In other words, a linear section in the `self-intersection' of $\Gr(2,5)$.

In general, we predict that the examples obtained by setting $P = P_{k_1,k_2}$ are various smoothing components of the join of anti-canonical sections in a del~Pezzo surfaces of degree $12-k_1$ and $12-k_2$ respectively. For example, considering pairs $(P_{6,7},D)$, we obtain smoothing components corresponding to the smoothing components of $\join(\Gr(2,5),\PP^2 \times \PP^2)$ and $\join(\Gr(2,5),\PP^1 \times \PP^1 \times \PP^1)$ obtained in \cite{I19} by Inoue. These correspond to standard decompositions $D$ in which every hexagonal face is decomposed into line segments ($\PP^2\times \PP^2$) or triangles ($\PP^1\times \PP^1\times \PP^1$). Similarly, the smoothing component of $\join(\Gr(2,5),\Bl_{pt}\PP^3)$ described in \cite{I19} corresponds -- in our context -- to $(P_{5,7},D)$, in which $D$ is the unique standard decomposition of $P_{5,7}$. We refer to Tables~\ref{tbl:P6k} and \ref{tbl:Psingle} for further examples.

In the tables below we analyse all possible examples of the form $(P_{k_1,k_2},D)$; and hence -- at least conjecturally -- describe all smoothing components of all possible joins of elliptic curves with degree between $3$ and $8$. While still a small subclass of s.d.\ polytopes, this class contains a significant number of new examples. For example, there are $72$ possible choices of $D$ for the polytopes $P_{5,6}$, which carry an action of the automorphism group $C_2 \times D_{12}$. In fact the $D_{12}$ factor acts trivially; hence the possible choices for $D$ are given by the orbits of a $C_2$ action. See Table~\ref{tbl:P65}.

We now consider the polytope $P := P_{6,6}$ in greater detail. Geometrically, an anti-canonical hypersurface in the toric variety $X_P$ associated to $P$ has $48$ singular points. Of these, $36$ are ordinary double points, while $12$ are given by the anti-canonical cone on the (rigid) del Pezzo surface $dP_6$. The latter singularities are known to admit a pair of smoothings (corresponding to embedding $dP_6$ in either $\PP^2\times\PP^2$ or $\PP^1\times\PP^1\times\PP^1$).

\begin{rem}
	It is not known when such singularities may be simultaneously smoothed, and it would be interesting to describe the notion of regularity given in Definition~\ref{dfn:regular} in this case in terms similar to the homological conditions of Friedman \cite{F91} and Tian \cite{T92}.	
\end{rem}

Since each hexagonal face admits a pair of possible Minkowski decompositions, there are $2^{12}$ possible choices of Minkowski decompositions $D$. Let $\cD$ denote this set of possible decompositions. The automorphism group of $P$ is the wreath product $G := D_{12} \wr C_2$, a group of order $288$. The group $G$ acts on $\cD$, resulting in a set $91$ orbits. To describe these orbits we index the vertices of $P_6$ with numbers $[6] := \{1,\ldots,6\}$. Each hexagonal face of $P_{6,6}$ is either equal to $v_i \times P_6$ or $P_6 \times v_i$ for some $i \in [6]$. We record $D \in \cD$ as a pair $(A_1,A_2)$ of subsets of $[6]$. The first of these contains indices $i$ such that $D$ assigns $v_i \times P_6$ its Minkowski decomposition into a pair of triangles. The second contains indices $i$ such that $D$ assigns $P_6 \times v_i$ its Minkowski decomposition into a pair of triangles.

Writing $n_i$ for the number of elements in $A_i$ for each $i \in \{1,2\}$, note that $G$ acts on the subsets of $\cD$ with fixed values of $n_1$ and $n_2$. We tabulate the number of orbits in each case, were $n_1$ and $n_2$ are given by the row and column index respectively in Table~\ref{tbl:orbits}.

\begin{table}
\[
\begin{array}{c|ccccccc}
& 0 & 1 & 2 & 3 & 4 & 5 & 6 \\ \hline
0 & 1 & 1 & 3 & 3 & 3 & 1 & 1 \\ 
1 & & 1 & 3 & 3 & 3 & 1 & 1 \\
2 & & & 6 & 9 & 9 & 3 & 3 \\
3 & & & & 6 & 9 & 3 & 3 \\
4 & & & & & 6 & 3 & 3 \\
5 & & & & & & 1 & 1 \\
6 & & & & & & & 1 \\
\end{array}
\]
\caption{Orbits of $G$ on $\cD$}
\label{tbl:orbits}
\end{table}
We now consider the topological invariants of the spaces $X(P,D)$ as $D$ ranges over all possible sets of Minkowksi decompositions in $\cD$.

\begin{pro}
	\label{pro:eg_invariants}
	Let $P := P_{6,6}$, and let $D \in \cD$ be determined by a pair $(A_1,A_2)$ of subsets of $[6]$. We have that $\chi(X(P,D)) = 2n_1+2n_2 - 72$. The second Betti number admits the following possibilities.
	\begin{enumerate}
		\item $b_2 = 5$ if $n_1=n_2=6$,
		\item $b_2 = 4$ if $\{n_1,n_2\}=\{0,6\}$,
		\item $b_2 = 3$ if $\{n_1,n_2\}=\{0,0\}$ or $\{k,6\}$ for $k \notin \{0,6\}$,
		\item $b_2 = 2$ if $\{n_1,n_2\}=\{0,k\}$ for $k \notin \{0,6\}$,
		\item $b_2 = 1$ in all other cases.
	\end{enumerate}
Thus we obtain $22$ topological types of $X(P,D)$; the invariants of which are displayed in Table~\ref{tbl:eg_invariants}.
\end{pro}
\begin{proof}
	The calculation of $\chi(X(P,D))$ is a direct application of Proposition~\ref{pro:euler_number}, similar to that appearing in Example~\ref{eg:ci33}.	We also compute the second Betti number of $X(P,D)$ using the approach taken in Example~\ref{eg:ci33}.
	
	Assigning variables $x_{i,j}$ and $x^{i,j}$ (the product of the $i$th edge and $j$th vertex, or $j$th vertex and $i$th edge) to each of $72$ edges for $i,j \in [6]$. We again observe that relations obtained from square faces of $P$ imply that $x_{i,j} = x_{i,k}$ for all $i, k \in [6]$. Hence we have a $12$ dimension solution space generated by variables $x^i$ and $x_i$ for $i \in [6]$ before applying relations coming from hexagonal faces. We have three cases.
	\begin{enumerate}
		\item If $D$ assigns every face $P_6 \times v_i$ for $i \in [6]$ the Minkowski decomposition into a triple of line segments, the relations imposed on variables $x_i$ have the form
		\begin{align*}
		x_1+x_4 = 0 && x_2+x_5 =0 && x_3+x_6=0. 
		\end{align*}
		\item If $D$ assigns every face $P_6 \times v_i$ for $i \in [6]$ the Minkowski decomposition into a pair of triangles, the relations imposed on variables $x_i$ have the form
		\begin{align*}
		x_1+x_3+x_5 = 0 && x_2+x_4+x_6=0. 
		\end{align*}
		\item If $D$ assigns at least one face $P_6 \times v_i$ for $i \in [6]$ to each possible Minkowski decomposition we impose all five of the above equations on $\{x_i : i \in [6]\}$. Note that these five equations are not independent, and the subspace of solutions to these equations has codimension $4$.
	\end{enumerate}
	Applying similar conditions to the variables $x^i$ for $i \in [6]$, we obtain the list presented in the statement of Proposition~\ref{pro:eg_invariants}. Note that this calculation is easy to generalise to any $(P_{k_1,k_2},D)$, a calculation used to compute entries in the columns titled $b_2$ in Appendix~\ref{sec:tables}.
\end{proof}

\begin{table}
\[
		\begin{array}{c|ccccccc}
		  & 0 & 1 & 2 & 3 & 4 & 5 & 6 \\ \hline
		0 & (3,-72) & (2,-70) & (2,-68) & (2,-66) & (2,-64) & (2,-62) & (4,-60) \\ 
		1 & & (1,-68) & (1,-66) & (1,-64) & (1,-62) & (1,-60) & (3,-58) \\
		2 & & & (1,-64) & (1,-62) & (1,-60) & (1,-58) & (3,-56) \\
		3 & & & & (1,-60) & (1,-58) & (1,-56) & (3,-54) \\
		4 & & & & & (1,-56) & (1,-54) & (3,-52) \\
		5 & & & & & & (1,-52) & (3,-50) \\
		6 & & & & & & & (5,-48) \\
		\end{array}
\]
		\caption{Invariants $(b_2,\chi)$ of $X(P,D)$ for each $D \in \cD$.}
		\label{tbl:eg_invariants}
\end{table}
So far our analysis has only concerned the topology of the space $X(P,D)$, which can done irrespectively of the regularity of $(P,D)$. We now determine which of these topological types can be realised with a regular pair $(P,D)$, and hence correspond to genuine Calabi-Yau threefolds. This is achieved using an implementation of the method described in \S\ref{sec:sd-poytopes} in computer algebra, and we include our source file as supplementary material. Note that, by Proposition~\ref{pro:irregular}, we only need to verify the existence of a strictly convex consistent slope function. Indeed, all $2$-dimensional faces $\sigma$ of $P^\circ$ are standard simplices and $\sigma^\star$ always has lattice length equal to one. The result of these computations is stated as Proposition~\ref{pro:regular_examples}.

\begin{pro}
	\label{pro:regular_examples}
	The pair $(P,D)$ is regular unless either $n_1$ or $n_2$ belongs to $\{1,5\}$.
\end{pro}

Geometrically, this means we can conjecture the existence (and non-existence) of certain smoothing components for anti-canonical sections of the toric variety $X_P$. Recall these have $12$ singularities, which we label
\[
\{p_1,\ldots p_6,p^1,\ldots,p^6\},
\]
equal to the cone on the toric del Pezzo surface $dP_6$. Locally we can choose one of two smoothing components at each of these $12$ singular points. Following Altmann~\cite{A97}, these are in canonical bijection with the Minkowski decompositions of a hexagon.

\begin{conjecture}
	Fix local smoothing components for the $12$ singularities of a general anti-canonical hypersurface $X \subset X_P$ locally isomorphic to the cone on $dP_6$. These choices determine a smoothing of $X$, simultaneously smoothing each of these $12$ singularities in the chosen component, unless precisely five of the six smoothing components chosen for the singularities $\{p_1,\ldots, p_6\}$ or $\{p^1,\ldots,p^6\}$ coincide. 
\end{conjecture}

\subsection{Hadamard products}
\label{sec:Hadamard}

Consider the pair $(P,D)$, where $P = P_{6,6}$, and $D$ is the set of Minkowski decompositions of all the $2$-dimensional faces of $P$ into line segments. We have seen that $(P,D)$ is regular, and that $\chi(X(P,D)) = -72$. If $b_2(X(P,D))$ were equal to $1$, we would have that $H^3 = 36$ for a generator $H$ of $H^2(X(P,D),\QQ)$. Moreover, we observe that -- for a suitable choice of log structure on $X_0(P,D)$ -- we expect smoothings to support a large automorphism group, inherited from the automorphism group of $P$. Such a (rank $1$) Calabi--Yau threefold was predicted by van Enckevort and van Straten \cite{ES03}. The prediction is via the construction of the Calabi--Yau differential operator
\begin{align}
\label{eq:operator}
	D :=~&\theta^4 - t(73\theta^4 + 98\theta^3 + 77\theta^2 + 28\theta + 4) + t^2(520\theta^4 - 1040\theta^3 - 2904\theta^2 - 2048\theta - 480) \nonumber \\ &+ 2^6t^3(65\theta^4 + 390\theta^3 + 417\theta^2 + 180\theta + 28)- 2^9t^4(73\theta^4 + 194\theta^3 + 221\theta^2 + 124\theta + 28) \\ &+ 2^{15}t^5(\theta + 1)^4 \nonumber
\end{align}

which has series solution,
\[
\sum_{n = 0}^\infty{\Big\{\sum_{k=0}^n{\binom{n}{k}}\Big\}^2t^n}.
\]
We note that $D$ is operator $100$ in the tables of Almkvist--van Enckevort--van Straten~\cite{AvEvSZ05}. This series solution is the \emph{Hadamard square} of a lower order operator; that is the coefficents $A_n$ of the series solution are themselves squares. Mirror symmetry considerations predict that the Hadamard product structure of the series solution of Picard--Fuchs operators of a Calabi--Yau threefold is compatible with the product structure of the polytopes $P$ associated to toric degenerations of this threefold: by way of a more straightforward example, the series solutions of the differential operator associated to a $(3,3)$ complete intersection in $\PP^5$ is equal to
\[
\sum_{n = 0}^\infty{\Big\{\frac{(3n)!}{n!^3}\Big\}^2z^n}.
\]

For more details on this connection, we refer to Galkin~\cite{G15}. To summarise, given a polygon $P_k$ for $k \in \{3,\ldots, 9\}$ or $8'$, we may form a Laurent polynomial $f_k$ by putting binomial coefficients along the edges. After making a judicious choice of constant term, the period sequences
\[
\pi_f(t) = \int_{\Gamma}{\frac{\Omega}{1-tf}},
\]
where $\Gamma$ is a compact torus in ${\Cstar}^2$ and $\Omega$ is the usual logarithmic volume form, satisfy Picard--Fuchs equations obtained in \cite{AZ06} in some cases. We identify these equations following the notation used in \cite{AZ06}.
\begin{enumerate}
	\item $\pi_{f_9}$ satisfies equation (B).
	\item $\pi_{f_8}$ and $\pi_{f_{8'}}$ satisfy equation (A).
	\item $\pi_{f_7}$ satisfies equation (b).
	\item $\pi_{f_6}$ satisfies either equation (a) or (c) (the equation (a) corresponds to the Minkowski decomposition into three line segments).
	\item $\pi_{f_5}$ is somewhat more mysterious. It satisfies an equation with $t$-degree $4$ when the constant term is set to $1$.
	\item $\pi_{f_4}$ satisfies equation (n) in \cite{AZ06}. As noted in \cite{AZ06} Hadamard products with (n) correspond to a substitution $t \mapsto t^{1/2}$.
\end{enumerate}

It follows immediately that the product of $f_{k_1}$ with $f_{k_2}$ (regarded as functions on different complex tori) satsifies the differential equation given by the Hadamard product of the operators for $f_{k_1}$ and $f_{k_2}$. Hence we predict that our methods recover a number of the examples listed by van Enckevort--van Straten in \cite{ES03}.

In particular, the above considerations suggest that the pair $(P,D)$ determines the smoothing of $X_{P}$ to the join of a pair of elliptic curves in $\PP^1\times \PP^1\times \PP^1$. In particular, we the above considerations provide evidence for the following conjecture concerning Calabi--Yau threefolds with $\chi = -72$.

\begin{conjecture}
	\label{thm:exists}
	The differential operator $D$ is associated to the invariant part of the rank $3$ Calabi--Yau threefold $X$ (with $\chi = -72$) obtained from the pair $(P,D)$ -- and suitable log structure -- with respect to a $\ZZ_2$ subgroup of $\Aut(X)$. This $\ZZ_2$ action is free and the quotient of $X$ by this action is the rank $1$ Calabi--Yau threefold with $h^{1,2} = 19$ obtained in \cite[\S$2.5$]{KS19}.
\end{conjecture}

\appendix
\section{Tables of Calabi--Yau threefolds}
\label{sec:tables}

In this appendix we describe collections of pairs $(P,D)$, where $P = P_{k_1,k_2}$ and $k_1$, $k_2$ range over $\{3,4,5,6,7,8,9\}$. Note that we suppress the value $8'$, since we do not expect it to provide Calabi--Yau threefolds we cannot obtain from polytopes of the form $P_{k_1,8}$.

Following \S\ref{sec:Hadamard} we make a number of predictions on the Picard--Fuchs operators for mirror families, based on the tables given in \cite{AvEvSZ05,AESZ:database}. Note that in some entries we specify operators using the names given in \cite{AZ06}.

The value $\Vol(P^\circ)$ is recorded as a proxy for $H^3$, where $H$ is a hyperplane section. In rank one cases $\Vol(P^\circ)$ is the cube of a generator of $H^2(X,\QQ)$ if this value is cube-free. If $\Vol(P^\circ)$ is not cube-free we check the various possible values of $H^3$ in lists of known Calabi--Yau threefolds. The columns titled `configuration' record how different Minkowski decompositions are chosen, while the columns titled 'orbits' contain the number of orbits for each specified configuration. We describe how to read the column `configuration' in each case.
\begin{enumerate}
	\item In Table~\ref{tbl:P6k} `configuration' records the number $n$ of hexagonal faces decomposed into a pair of triangles. Hence $(k-n)$ faces are decomposed into a triple of line segments.
	\item In Table~\ref{tbl:P66} `configuration' records pairs $(n_1,n_2)$. We assume $n_1$ faces of the form $\{\star\} \times P_6$ are decomposed into a pair of triangles, and $n_2$ faces of the form $P_6 \times \{\star\}$. This case is discussed in greater detail in \S\ref{sec:examples}.
	\item In Table~\ref{tbl:P65} `configuration' records pairs $(n_1,n_2)$ together with values $m \in \{0,1,2,3\}$. The edge of $P^\circ_{5,6}$ dual to two of the $5$ hexagonal facets has length two; hence we must choose one of three possible Minkowski decompositions of $2 \times P_6$ for each of these two faces. These decompositions are stored as $n_1$ and $n_2$ respectively. The value $m$ stores the number of the remaining $3$ hexagonal faces which are decomposed into a pair of triangles.
	\item In Table~\ref{tbl:P64} `configuration' records tuples $(n_1,n_2,n_3,n_4)$. All four hexagonal faces of $P_{4,6}$ are dual to edges of length $2$. Hence we must choose one of $3$ Minkowski decompositions of $2 \times P_6$ for each such hexagonal face. These are stored in $n_i$ for each $i$, where the vertices of $P_4$ are labelled in a clockwise direction.
\end{enumerate}

Entries $n_i$, $i \in \{1,2,3,4\}$ in Table~\ref{tbl:P65} and Table~\ref{tbl:P64} label each hexagonal face dual to a length two line segment with a value in $\{0,1,2\}$. The value $0$ corresponds to the Minkowski decomposition of $2 \times P_6$ into six line segments; the value $1$ corresponds to the decomposition into three line segments and two triangles; and the value $2$ corresponds to the decomposition into four triangles.

\begin{rem}
	We conclude with a remark on the completeness of the tables. The table rows correspond to collections of orbits. The presence of a row means we found some representative $D$ in at least one such orbit such that $(P,D)$ is regular. Moreover, the absence of a configuration in a table means there is at least one pair $(P,D)$ of the given configuration for which regularity is not satisfied. We have not checked every orbit in every class, so it remains possible that our tables are incomplete.
\end{rem}

\begin{table}[htb]
	\[\def\arraystretch{1.5}
	\begin{array}{ccccccc}
	(k_1,k_2) & \chi & b_2 & \Vol(P^\circ) & \text{AESZ} & \#\text{Orbits} & \text{Configuration} \\\hline
	
	(6,7) & -90 & 2 & 30 & 102 & 1 & (0) \\
	(6,7) & -86 & 1 & 30 & & 6 & (2) \\
	(6,7) & -84 & 1 & 30 & & 6 & (3) \\
	(6,7) & -80 & 3 & 30 & 113 & 1 & (5) \\
	(6,8) & -120 & 2 & 24 & 45 & 1 & (0) \\
	(6,8) & -116 & 1 & 24 & 29 & 2 & (2) \\
	(6,8) & -112 & 3 & 24 & 58 & 1 & (4) \\
	(6,9) & -162 & 2 & 18 & 15 & 1 & (0) \\
	(6,9) & -156 & 3 & 18 & 70 & 1 & (3) \\
	\end{array}
	\]
	\caption{Calabi--Yau threefolds from regular pairs $(P_{6,k},D)$, $k \in \{8,9\}$}
	\label{tbl:P6k}
\end{table}

\begin{table}
	\[\def\arraystretch{1.5}
	\begin{array}{ccccccc}
	(k_1,k_2) & \chi & b_2 & \Vol(P^\circ) & \text{AESZ} & \#\text{Orbits} & \text{Configuration} \\\hline
	(6,6) & -64 & 1 & 36 & & 6 & (2,2) \\
	(6,6) & -62 & 1 & 36 & & 6 & (2,3) \\
	(6,6) & -60 & 1 & 36 & & 9 & (2,4) \\
	(6,6) & -60 & 1 & 36 & & 6 & (3,3) \\
	(6,6) & -58 & 1 & 36 & & 9 & (3,4) \\
	(6,6) & -56 & 1 & 36 & & 6 & (4,4) \\	
	(6,6) & -68 & 2 & 36 & & 3 & (0,2) \\
	(6,6) & -66 & 2 & 36 & & 3 & (0,3) \\
	(6,6) & -64 & 2 & 36 & & 3 & (0,4) \\
	(6,6) & -72 & 3 & 36 & 100 & 1 & (0,0) \\
	(6,6) & -56 & 3 & 36 & & 3 & (2,6) \\
	(6,6) & -54 & 3 & 36 & & 3 & (3,6) \\
	(6,6) & -52 & 3 & 36 & & 3 & (4,6) \\
	(6,6) & -60 & 4 & 36 & 104 & 1 & (0,6) \\
	(6,6) & -48 & 5 & 36 & 103 & 1 & (6,6) \\
	\end{array}
	\]
	\caption{Calabi--Yau threefolds from regular pairs $(P_{6,6},D)$}
	\label{tbl:P66}
\end{table}

\begin{table}
	\[\def\arraystretch{1.5}
	\begin{array}{ccccccc}
	(k_1,k_2) & \chi & b_2 & \Vol(P^\circ) & \#\text{Orbits} & \text{Configuration} \\\hline	
	(6,5) & -90 & 3 & 42 & 1 & (0,0),0 \\
	(6,5) & -86 & 2 & 42 & 2 & (0,0),2 \\
	(6,5) & -84 & 2 & 42 & 1 & (0,0),3 \\
	(6,5) & -86 & 2 & 42 & 3 & (0,1),1 \\
	(6,5) & -84 & 2 & 42 & 3 & (0,1),2 \\
	(6,5) & -82 & 2 & 42 & 1 & (0,1),3 \\
	(6,5) & -86 & 2 & 42 & 1 & (1,1),0 \\
	(6,5) & -84 & 2 & 42 & 2 & (1,1),1 \\
	(6,5) & -82 & 2 & 42 & 2 & (1,1),2 \\
	(6,5) & -80 & 2 & 42 & 1 & (1,1),3 \\
	(6,5) & -86 & 2 & 42 & 1 & (0,2),0 \\
	(6,5) & -84 & 2 & 42 & 3 & (0,2),1 \\
	(6,5) & -82 & 2 & 42 & 3 & (0,2),2 \\
	(6,5) & -80 & 2 & 42 & 1 & (0,2),3 \\
	(6,5) & -84 & 2 & 42 & 1 & (1,2),0 \\
	(6,5) & -82 & 2 & 42 & 3 & (1,2),1 \\
	(6,5) & -80 & 2 & 42 & 3 & (1,2),2 \\
	(6,5) & -78 & 2 & 42 & 1 & (1,2),3 \\
	(6,5) & -82 & 2 & 42 & 1 & (2,2),0 \\
	(6,5) & -80 & 2 & 42 & 2 & (2,2),1 \\
	(6,5) & -76 & 4 & 42 & 1 & (2,2),3 \\
	\end{array}
	\]
	\caption{Calabi--Yau threefolds from regular pairs $(P_{5,6},D)$}
	\label{tbl:P65}
\end{table}

\begin{table}
	\[\def\arraystretch{1.5}
	\begin{array}{ccccccc}
	(k_1,k_2) & \chi & b_2 & \Vol(P^\circ) & \text{Operator} & \#\text{Orbits} & \text{Configuration} \\\hline
	
	(6,4) & -72 & 2 & 48 & \text{(n)}\star \text{(a)} & 1 & (0,0,0,0) \\
	(6,4) & -68 & 1 & 48 & & 2 & (0,0,1,1) \\
	(6,4) & -66 & 1 & 48 & & 1 & (0,1,1,1) \\
	(6,4) & -64 & 1 & 48 & & 1 & (1,1,1,1) \\
	(6,4) & -68 & 1 & 48 & & 1 & (0,0,0,2) \\
	(6,4) & -64 & 1 & 48 & & 2 & (0,0,2,2) \\
	(6,4) & -60 & 1 & 48 & & 1 & (0,2,2,2) \\
	(6,4) & -56 & 3 & 48 & \text{(n)}\star \text{(c)} & 1 & (2,2,2,2) \\
	(6,4) & -66 & 1 & 48 & & 2 & (0,0,1,2) \\
	(6,4) & -64 & 1 & 48 & & 2 & (0,1,1,2) \\
	(6,4) & -62 & 1 & 48 & & 1 & (1,1,1,2) \\	
	(6,4) & -62 & 1 & 48 & & 2 & (0,1,2,2) \\
	(6,4) & -60 & 1 & 48 & & 2 & (1,1,2,2) \\
	\end{array}
	\]
	\caption{Calabi--Yau threefolds from regular pairs $(P_{4,6},D)$}
	\label{tbl:P64}
\end{table}

\begin{table}
	\[\def\arraystretch{1.5}
	\begin{array}{cccccc}
	(k_1,k_2) & \chi & b_2 & \Vol(P^\circ) & \text{Operator} & \text{Note} \\\hline
	(4,4) & -64 & 1 & 64 & \text{(n)}\star \text{(n)} &  \\
	(4,5) & -60 & 2 & 56 & & \\
	(4,7) & -100 & 1 & 40 & \text{(b)}\star \text{(n)} & \\
	(4,8) & -144 & 1 & 2^3\cdot 4 &\text{(A)}\star \text{(n)}  & X_{4,4} \subset \PP(1^4,2^2) \\
	(4,9) & -204 & 1 & 2^3\cdot 3 & \text{(B)}\star \text{(n)}  & X_6 \subset \PP(1^4,2) \\
	(5,5) & -56 & 3 & 49 & & \\
	(5,7) & -90 & 2 & 35 & & \\
	(5,8) & -128 & 2 & 28 & & \\
	(5,9) & -180 & 2 & 21 & & \\
	(7,7) & -100 & 1 & 25 & 101 & \Gr(2,5) \cap \Gr(2,5) \\
	(7,8) & -120 & 1 & 20 & 25 & X(1,2,2) \subset \Gr(2,5) \\
	(7,9) & -150 & 1 & 15 & 24 & X(1,1,3) \subset \Gr(2,5) \\
	(8,8) & -128 & 1 & 16 & 3 & X_{2,2,2,2} \subset \PP^7 \\
	(8,9) & -144 & 1 & 12 & 5 & X_{2,2,3} \subset \PP^6 \\
	(9,9) & -144 & 1 & 9 & 4 & X_{3,3} \subset \PP^5 \\
	\end{array}
	\]
	\caption{Calabi--Yau threefolds from regular pairs $(P_{k_1,k_2},D)$, $k_1,k_2 \neq 6$}
	\label{tbl:Psingle}
\end{table}

\clearpage
\bibliographystyle{plain}
\bibliography{bibliography}
\end{document}